\documentclass[reqno,12pt]{amsart}
\usepackage{amsmath,amssymb,graphics}
\newtheorem{theorem}{Theorem}[section]
\newtheorem{proposition}[theorem]{Proposition}
\newtheorem{definition}[theorem]{Definition}
\newtheorem{lemma}[theorem]{Lemma}
\newtheorem{corollary}[theorem]{Corollary}

\numberwithin{equation}{section}
\numberwithin{theorem}{section}

\newcommand{\mc}[1]{{\mathcal #1}}
\newcommand{\mf}[1]{{\mathfrak #1}}
\newcommand{\mb}[1]{{\mathbf #1}}
\newcommand{\bb}[1]{{\mathbb #1}}
\newcommand{\eps}{\varepsilon}
\newcommand{\upbar}[1]{\,\overline{\! #1}}
\newcommand{\id}{{1 \mskip -5mu {\rm I}}}

\begin{document}

\title{Large Deviations Principles for Stochastic Scalar Conservation
  Laws}

\author{Mauro Mariani} 
\address{M. Mariani \\ CEREMADE, UMR-CNRS
  7534, Universit\'e de Paris-Dauphine, Place du Marechal de Lattre de
  Tassigny, F-75775 Paris Cedex
  16.} 
\email{mariani@ceremade.dauphine.fr}

\maketitle

\begin{abstract}
  Large deviations principles for a family of scalar $1+1$ dimensional
  conservative stochastic PDEs (viscous conservation laws) are
  investigated, in the limit of jointly vanishing noise and
  viscosity. A first large deviations principle is obtained in a space
  of Young measures. The associated rate functional vanishes on a wide
  set, the so-called set of measure-valued solutions to the limiting
  conservation law. A second order large deviations principle is
  therefore investigated, however, this can be only partially
  proved. The second order rate functional provides a generalization
  for non-convex fluxes of the functional introduced by Jensen
  \cite{J} and Varadhan \cite{V} in a stochastic particles system
  setting.

%\keywords{Stochastic PDE \and Large Deviations
%    \and Conservation Laws \and Entropy functional}
%\PACS{05.40.-a \and 02.50.Ey}
%\subclass{60H15 \and 60F10 \and 60K35}
\end{abstract}

\section{Introduction}
Macroscopic description of physical systems with a large number of
degrees of freedom can be often provided by the means of partial
differential equations.  Rigorous microscopic derivations of such PDEs
have been proved in different settings, and we will refer in
particular to stochastic interacting particles systems \cite{KL,S},
where stochastic microscopic dynamics of particles are considered. One
is usually interested in the asymptotic properties of the empirical
measures associated with some relevant physical quantities of the
system, such as the particles density. Provided that time and space
variables are suitably rescaled, it has been proved for several models
that, as the number of particles diverges to infinity, the empirical
measure associated with the particles density converges to a
``macroscopic density'' $u \equiv u(t,x)$.  Moreover such a density
$u$ solves a limiting ``hydrodynamical equation'', which in the
conservative case has usually the following structure
\begin{eqnarray}
\label{e:1.1}
 \partial_t u +\nabla \cdot \big(f(u)-D(u) \nabla u \big)=0
\end{eqnarray}
Here $\nabla$ and $\nabla \cdot $ stands for
the space gradient and divergence operators, $D \ge 0$ is a diffusion
coefficient, while the flux $f$ takes into account the transport phenomena that
may occur in the system. Roughly speaking, $D$ is strictly positive for
symmetric (or zero mean) and weakly asymmetric systems, in which case
\eqref{e:1.1} is usually
obtained in the so-called \emph{diffusive scaling} of the time and space
variables. The case $D \equiv 0$ is instead associated with asymmetric systems,
and is usually obtained in the so-called \emph{Euler scaling}.

Once the hydrodynamics of the density is understood, a deeper insight
into the system behavior is provided by the investigation of large
deviations for the probability law of the empirical measure associated
with the density. Establishing large deviations for these models can
in fact provide a better understanding of the concepts of entropy and
fluctuations in the context of non-equilibrium statistical
mechanics. However, while several large deviations results have been
obtained for symmetric (or weakly asymmetric) systems under diffusive
scaling \cite{KL}, very little is known for asymmetric systems, with
the remarkable exception of the seminal works \cite{L,J,V}. According
to \cite[Chap.~8]{KL}, large deviations for asymmetric processes are
``one of the main open questions in the theory of hydrodynamical
limits''.

\subsection{Stochastic conservation laws} In this paper we will focus
on a slightly different approach. We consider a continuous
``mesoscopic density'' $u^\eps \equiv u^\eps(t,x) \in \bb R$ depending
on a small parameter $\eps$ (which should be regarded as the inverse
of the number of particles). We assume that $u^\eps$ satisfies a
continuity equation, with a stochastic current taking into account the
transport, diffusion and fluctuation phenomena that may occur in the
system. More precisely, for $\eps,\gamma>0$ we consider the stochastic
PDE in the unknown $u$
\begin{eqnarray}
\label{e:1.2}
 \partial_t u+ \nabla \cdot \big(f(u)-\frac{\eps}2 D(u)
\nabla u
-\eps^\gamma\,\sqrt{a^2(u)}\,\alpha^\eps \big)=0
\end{eqnarray}
where $a^2$ is a fluctuation coefficient, and $\alpha^\eps$ is a
stochastic noise, white in time and with a correlation in space
regulated by a convolution kernel $\jmath^\eps$. We assume that
$\jmath^\eps$ converges to the identity as $\eps \to 0$, namely that
the the range of spatial correlations vanishes at the macroscopic
scale. We are then interested in the asymptotic properties
(convergence and large deviations) of the solution $u^\eps$ to
\eqref{e:1.2}, as $\eps \to 0$, namely as diffusion and noise vanish
\emph{simultaneously}. We remark that, while equations of the form
\eqref{e:1.2} may describe quite general physical systems, the limit
$\eps \to 0$ is indeed motivated by the heuristic behavior of the
density of asymmetric particles systems under Euler scaling. In fact,
while one expects the stochastic noise and its spatial correlation to
vanish at a macroscopic scale for quite general systems, the limit of
jointly vanishing viscosity and noise is somehow specific for the
Euler scaling. This specific feature may be one of the (several)
reasons making the large deviations of asymmetric systems more
challenging.

From the point of view of stochastic PDEs, the limit $\eps \to 0$ also
introduces new difficulties. In fact, large deviations for diffusion
processes have been widely investigated \cite{FW,DZ} in the vanishing
noise case, and general methods are available to identify the rate
functionals associated with large deviations. On the other hand, at
our knowledge no results are available -even for finite dimensional
diffusions- if vanishing noise and deterministic drift with nontrivial
limiting behavior are considered (here the deterministic drift has a
so-called singular limit, see \eqref{e:1.4}). As shown below, in this
more general case one needs to investigate a (deterministic)
variational problem associated with the stochastic equation. The
variational problem associated to \eqref{e:1.2} has been addressed in
\cite{BBMN} in a slightly different setting, and we will use most of
the results therein obtained.

With respect to the models usually considered in particles systems,
\eqref{e:1.2} allows us to get rid of several technicalities related
to the discrete nature of particles; we may thus provide a unified
treatment of several models (that is, $f$, $D$ and $a$ are
arbitrary). However, as discussed below, the results obtained (namely
the speed and rates of large deviations) are in substantial agreement
with \cite{J,V} if the case $f(u)=a^2(u)=u(1-u)$ and $D(u)=1 $ is
considered.

\subsection{Outline of the results}
Informally setting $\eps=0$ in \eqref{e:1.2}, we obtain the
deterministic PDE
\begin{eqnarray}
\label{e:1.3}
 \partial_t u+ \nabla \cdot f(u) =0
\end{eqnarray}
usually referred to as a \emph{conservation law}. As well known
\cite[Chap.~4]{Daf}, if $f$ is nonlinear, the Cauchy problem associated to
\eqref{e:1.3} does not admit global smooth solutions, even if the
initial datum is smooth. In general there exist infinitely many weak
solutions to \eqref{e:1.3}, and an additional \emph{entropic
  condition} is needed to recover uniqueness and to identify the
relevant physical weak solution to \eqref{e:1.3}. While \eqref{e:1.3}
is invariant under the transformation $(t,x) \mapsto (-t,-x)$, the
entropic condition selects a direction of the time, by requiring that
entropy is dissipated. A classical result in PDE theory states that
the solution to
\begin{eqnarray}
\label{e:1.4}
 \partial_t u+ \nabla \cdot \big(f(u)-\frac{\eps}2
D(u)\nabla u \big)=0
\end{eqnarray}
converges to the entropic solution to \eqref{e:1.3} as $\eps \to 0$,
provided the initial data also converge. At the heuristic level, the
entropic condition keeps memory of the diffusive term in \eqref{e:1.4}
which indeed breaks the symmetry $(t,x) \mapsto (-t,-x)$. We will
briefly recall the definition of entropic \emph{Kruzkov} solutions to
\eqref{e:1.3} in Section~\ref{s:2}, and refer to \cite{Daf} for an
introduction to conservation laws.

There is only a few literature for existence and uniqueness of
solutions to fully nonlinear stochastic parabolic equations, see e.g.\
\cite{LS1} and \cite{LS2} dealing with finite-dimensional noise. Under
general hypotheses, in the appendix we provide existence and
uniqueness (for $\eps$ small enough and $\gamma>1/2$) for the Cauchy
problem associated to \eqref{e:1.2}, by the means of a piecewise
semilinear approximation of such equation. In Section~\ref{ss:3.1} we
gather some a priori bounds for the solution $u^\eps$ to
\eqref{e:1.2}, and show that, as $\eps \to 0$, $u^\eps$ converges in
probability to the entropic Kruzkov solution to \eqref{e:1.3} in a
strong topology.

We next analyze large deviations principles for the law of $u^\eps$ as
$\eps \to 0$. In order to avoid technical difficulties associated with
the unboundedness of $u^\eps$, and in order to keep our setting as
close as possible to the one considered in \cite{J,V}, we assume that
the fluctuation coefficient $a^2(u)$ vanishes for $u \not \in
(0,1)$. As we will also assume the initial datum to take values in
$[0,1]$, this condition guarantees that $u^\eps$ takes values in
$[0,1]$, see Theorem~\ref{t:exun}. We only consider the $(1+1)$
dimensional case, with the $(t,x)$ variables running in $[0,T]\times
\bb T$, where $T>0$ and $\bb T$ is the one dimensional torus.  While
these restrictions are merely technical, we remark that only the case
of scalar $u$ is considered, as the vectorial case (systems of
conservation laws) is certainly far more difficult.

In Section~\ref{ss:3.2} we establish a large deviations principle with
speed $\eps^{-2\gamma}$, roughly equivalent to the classical
Freidlin-Wentzell speed for finite dimensional diffusions
\cite{FW}. The bottom line is that, when events with probability of
order $e^{\eps^{-2\gamma}}$ are considered, the noise term in
\eqref{e:1.2} can bitterly deviate from its ``typical behavior'' thus
completely overcoming the regularizing effect of the vanishing
parabolic term. Any entropy-dissipation phenomena is lost at this
speed, and the noise may drive severe oscillations of the density
$u^\eps$ as $\eps \to 0$. The large deviations are then naturally
investigated in a Young measures setting. We prove that on a Young
measure $\mu\equiv \mu_{t,x}(d\lambda)$ (satisfying a suitable initial
condition) the large deviations rate functional is given by (see
Section~\ref{ss:2.4} for a more precise definition of $\mc I$)
\begin{eqnarray*}
\mc I (\mu) :=  \frac 12 \int_0^T\!dt \,
               \Big\|  \partial_t \mu(\imath)
                 +\nabla \cdot \mu( f)
               \Big\|_{H^{-1}(\mu (a^2), dx)}^2
\end{eqnarray*}
Here $\imath:\bb R \to \bb R$ is the identity map, for $F$ a
continuous function, $\mu(F)(t,x)$ stands for $\int
\mu_{t,x}(d\lambda) F(\lambda)$, and with a little abuse of notation,
we denoted by $\| \varphi \|_{H^{-1}( \mu_{t,\cdot} (a^2), dx)}$ the
dual norm to $\big[\int\!dx\, \mu_{t,x}(a^2) \,
\varphi_x^2\big]^{1/2}$.

Note that $\mc I(\mu)=0$ iff $\mu$ is a measure-valued solution to
\eqref{e:1.3} (see Section~\ref{ss:2.4}). The Cauchy problem
\eqref{e:1.3} admits in general infinitely many measure-valued
solutions, but we stated above that $u^\eps$ converges in probability
to the (unique) entropic solution to \eqref{e:1.3}. One thus expects
that nontrivial large deviations principle may hold with a speed
slower than $\eps^{-2\gamma}$. In Section~\ref{ss:3.3}, we investigate
large deviations principle with speed $\eps^{-2\gamma+1}$. At this
scale, deviations of the noise term in \eqref{e:1.3} are of the same
order of the parabolic term. The law of $u^\eps$ is then exponentially
tight (with speed $\eps^{-2\gamma+1}$) in a suitable space of
functions. To informally define the candidate rate functional for the
large deviations with this speed, we briefly introduce some
preliminary notions, which will be precisely explained in
Section~\ref{ss:2.5}.

We say that a weak solution $u$ to \eqref{e:1.3} is an
\emph{entropy-measure} solution iff there exists a measurable map
$\varrho_u$ from $[0,1]$ to the set of Radon measures on $(0,T)\times
\bb T$, such that for each $\eta \in C^2([0,1])$ and $\varphi \in
C^\infty_{\mathrm{c}}\big((0,T)\times \bb T \big)$
\begin{eqnarray*}
-\int\!dt\,dx\,\big[\eta(u) \varphi_t+ q(u)
\nabla \varphi \big] =
\int\!dv\,\varrho_u(v;dt,dx) \eta''(v) \varphi(t,x)
\end{eqnarray*}
where $q(v):=\int^v\!dw\,\eta'(w)f'(w)$, see Proposition~\ref{p:kin}
for a characterization of entropy-measure solutions to
\eqref{e:1.3}. The candidate rate functional for the second order
large deviations is the functional $H$ defined as follows. If $u$ is
not an entropy-measure solution to \eqref{e:1.3} then
$H(u)=+\infty$. Otherwise $H(u)=\int\,dv\,\varrho_u^+(v;dt,dx)
D(v)\,a^{-2}(v)$, where $\varrho_u^+$ denotes the positive part of
$\varrho_u$. Note that $H$ depends on the diffusion coefficient $D$
and the fluctuation coefficient $a^2$ only through their ratio, thus
fitting in the \emph{Einsten paradigm} for macroscopic diffusive
systems. We also remark that, while the functional $\mc I$ is convex,
$H$ is not (for instance, convex combinations of entropy-measure
solutions to \eqref{e:1.3}, in general are not weak solutions).

While we prove a large deviations upper bound with speed
$\eps^{-2\gamma+1}$ and rate $H$, we obtain the lower bound only on a
suitable set $\mc S$ of weak solutions to \eqref{e:1.3}, see
Definition~\ref{d:splittable}. To complete the proof of this
\emph{second order} large deviations, an additional density argument
is needed. This seems to be a challenging problem, and as noted by
Varadhan in \cite{V} ``\ldots one does not see at the moment how to
produce a `general' non-entropic solution, partly because one does not
know what it is.''

It is easy to see that, on the set of weak solutions to \eqref{e:1.3}
with bounded variations and on the set $\mc S$, the rate functional
$H^{JV}$ introduced in \cite{J,V} coincides with the rate functional
$H$ evaluated for $f(v)=v(1-v)$, $D \equiv 1$ and $a^2(v)=v(1-v)$,
which are the expected transport, diffusion and fluctuation
coefficients for the totally asymmetric simple exclusion process there
investigated. In particular, $H$ comes as a natural generalization of
the functional introduced in \cite{J,V}, whenever the flux $f$ is
neither convex nor concave. Unfortunately, since chain rule formulas
are not available out of the BV setting, one cannot check that
$H=H^{JV}$ on the whole set of entropy-measure solutions to
\eqref{e:1.3}.  Note however that the inequality $H\ge H^{JV}$
holds. Furthermore, under smoothness and genuine nonlinearity
assumption on $f$, $H(u)=0$ iff $u$ is the unique entropic solution to
\eqref{e:1.3}, so that higher order large deviations principles are
trivial.

\subsection{Outline of the proof}
The convergence in probability of $u^\eps$ to the entropic solution of
\eqref{e:1.3} is obtained by a sharp stability analysis of the
stochastic perturbation \eqref{e:1.2} of \eqref{e:1.4}.

The large deviations upper bound with speed $\eps^{-2\gamma}$ is
provided by lifting the standard Varadhan's minimax method to the
Young measures setting, while exponential tightness in this space is
easily proved.  The corresponding lower bound is first proved for
Young measures that are Dirac masses at almost every point $(t,x) \in
[0,T]\times \bb T$, and then extended to the whole set of Young
measures by adapting the relaxation argument in \cite{BBMN}.

The large deviations with speed $\eps^{-2\gamma+1}$ are much different
than the usual small noise asymptotic limit for It\^{o}
processes. Note indeed that, as $\eps \to 0$, the parabolic term in
\eqref{e:1.3} has a nontrivial behavior. In such a case there is no
general method to study large deviations, even in a finite dimensional
setting. We provide a link of the large deviations problem with a
$\Gamma$-convergence result obtained in \cite{BBMN}. Indeed we use the
equicoercivity of a suitable family of functionals to show exponential
tightness, and we use the so-called $\Gamma$-limsup result to build up
the optimal exponential martingales for the lower bound. In
particular, since the $\Gamma$-limsup inequality in \cite{BBMN} is not
fully established, we only have partial results for the lower
bound. The upper bound is established by a nonlinear version of the
Varadhan's minimax method.

\section{Main results}
\label{s:2}

\subsection{Notation}
\label{ss:2.1}

In this paper, $T>0$ is a positive real number and we let
$\big(\Omega, \mf F, \{\mf F_t\}_{0\le t \le T}, P\big)$ be a
\emph{standard} filtered probability space. For $B$ a real Banach
space and $M:[0,T]\times \Omega \to B$ a given adapted process, we
write equivalently $M(t)\equiv M(t,\omega)$. For each $\phi \in
B^\ast$ we denote by $\langle M,\phi \rangle \equiv \langle M,\phi
\rangle (t,\omega)$ the real--valued process obtained by the dual
action of $M$ on $B$. Given two real--valued $P$-square integrable
martingales $M,\,N$, we denote by $\big[M,N\big]\equiv
\big[M,N\big](t,\omega)$ the cross quadratic variation process of $M$
and $N$. In the following \emph{martingale} will always stand for
\emph{continuous martingale}. For a Polish space $X$, we also let $\mc
P(X)$ denote the set of Borel probability measures on $X$. For $\nu$ a
measure on some measurable space and $F \in L_1(d\nu)$, we denote by
$\nu(F)$ the integral of $F$ with respect to $\nu$. However, for a
probability $P$ we used the notation $\bb E^P$ to denote the expected
value.

We denote by $\bb T$ the one-dimensional torus, by $\langle
\cdot,\cdot\rangle$ the inner product in $L_2(\bb T)$, and by
$\langle\langle \cdot,\cdot\rangle\rangle$ the inner product in
$L_2([0,T]\times\bb T)$. For $E$ a closed set in $[0,T]\times \bb T$,
$C^k(E)$ denotes the collection of $k$-times differentiable functions
on $E$, with continuous derivatives up to the boundary. We also let
$H^1(\bb T)$ be the Hilbert space of square integrable functions on
$\bb T$ with square integrable derivative, and let $H_{-1}(\bb T)$ be
its dual space. Throughout this paper $\partial_t$ denotes derivative
with respect to the time variable $t$, $\nabla$ and $\nabla \cdot$
derivatives with respect to the space variable $x$ (while we consider
a one dimensional space setting, we consider gradient and divergence
as distinct operators). For a function $\vartheta$ explicitly
depending on the $x$ variable, $\partial_x$ denotes the partial
derivative with respect to $x$. Namely, given a function $u: \bb T \to
[0,1]$ and $\vartheta:[0,1]\times \bb T \to \bb R$, we understand
$\nabla [\vartheta(u(x),x)]= (\partial_u \vartheta)(u(x),x) \nabla
u(x) + (\partial_x \vartheta) (u(x),x)$. In the following we will
usually omit the dependence on the $\omega$ variable, as well as on
the $t$ and/or $x$ variables when no misunderstanding is possible.

\subsection{Stochastic conservation laws}
\label{ss:2.2}
We refer to \cite{DZ} for a general theory of stochastic differential
equations in infinite dimensions. Let $W$ be an $L_2(\bb T)$--valued
cylindrical Brownian motion on $\big(\Omega, \mf F, \{\mf F_t\}_{0\le
  t \le T}, P\big)$. Namely, $W$ is a Gaussian, $L_2(\bb T)$--valued
$P$-martingale with quadratic variation:
\begin{equation}
\label{e:2.1} 
\big[\langle W, \phi \rangle, \langle W , \psi \rangle \big](t,\omega)
  = \langle \phi, \psi \rangle\, t
\end{equation}
for each $\phi, \psi \in L_2(\bb T)$.

For $\eps >0$, we consider the following stochastic Cauchy problem in
the unknown $u$:
\begin{eqnarray}
\label{e:2.2}
\nonumber
& & d u = \big[- \nabla \cdot f(u)
+\frac{\eps}{2} \nabla \cdot \big(D(u) \nabla
u \big)\big]\,dt +\eps^{\gamma}\, \nabla \cdot
\big[a(u) (\jmath^\eps \ast d W)\big]
\\
& & u(0,x) = u^\eps_0(x)
\end{eqnarray}
Here $\gamma>0$ is a real parameter, and $\nabla \cdot \big[a(u)
(\jmath^\eps \ast d W)\big]$ stands for the martingale differential
acting on $\psi \in H^1(\bb T)$ as
\begin{eqnarray*}
  \big \langle \nabla \cdot
  \big[a(u) (\jmath^\eps \ast d W)\big], \psi \rangle 
  = - \langle d W, \jmath^\eps \ast [ a(u) \nabla \psi] \rangle
\end{eqnarray*}

The following hypotheses will be always assumed below, but in the
appendix.
\begin{itemize}
\item[\textbf{H1)}]{$f:[0,1]\to \bb R$ is a Lipschitz function.}
\item[\textbf{H2)}]{$D:[0,1]\to \bb R$ is a uniformly positive
    Lipschitz function.}
\item[\textbf{H3)}]{$a \in C^2([0,1])$ is such that $a(0)=a(1)=0$,
    and $a(v) \neq 0$ for $v \in (0,1)$.}
\item[\textbf{H4)}]{$\{\jmath^\eps\}_{\eps>0} \subset H^1(\bb T)$ is a
    sequence of positive mollifiers with $\int\!dx\,
    \jmath^\eps(x)=1$, weakly converging to the Dirac mass centered at
    $0$.}
\item[\textbf{H5)}]{For $\eps>0$, $u^\eps_0:\Omega \times \bb T \to
    [0,1]$ is a measurable map with respect to the product $\mf{F}_0$
    $\times$ Borel $\sigma$-algebra.  Moreover there exists a Borel
    measurable function $u_0:\bb T \to [0,1]$ such that, for each
    $\delta>0$
\begin{eqnarray*}
 \lim_\eps  P\big(\|u^\eps_0 -u_0\|_{L_1(\bb T)}>\delta \big)=0
\end{eqnarray*}
}
\end{itemize}

The next proposition is an immediate consequence of
Proposition~\ref{p:uunique} in the appendix, where we also recall the
precise definitions of strong and martingale solutions to
\eqref{e:2.2} and we briefly discuss why the condition on $\gamma$ and
$\jmath^\eps$ (see Proposition~\ref{t:uunique} below) are needed.
\begin{proposition}
\label{t:uunique}
Assume $\lim_\eps \eps^{2\gamma-1} \|\jmath^\eps\|_{L_2(\bb T)}^2
=0$. Then there is an $\eps_0>0$ depending only on $D$ and $a$, such
that, for each $\eps<\eps_0$, there exists a unique adapted process
$u^\eps: \Omega \to C\big([0,T];H^{-1}(\bb T)\big) \cap
L_2\big([0,T];H^1(\bb T)\big)$ solving \eqref{e:2.2} in the strong
stochastic sense. Moreover $u^\eps$ admits a version in
$C\big([0,T];L_1(\bb T)\big)$, and for every $t\in [0,T]$
$u^\eps(\omega;t,x) \in [0,1]$ for $d P\,dx$ a.e.\ $(\omega,\,x)$.
\end{proposition}
Note that the total mass of $u^\eps$ is conserved a.s.\ by the
stochastic flow \eqref{e:2.2}, namely for each $t \in [0,T]$ we have
$\int \!dx\, u^\eps(t,x)=\int \!dx\, u_0^{\eps}(x)$ $P$ a.s.. We are
interested in the asymptotic limit of the probability law of the solution
$u^\eps$ to \eqref{e:2.2} as $\eps \to 0$.

\subsection{Deterministic conservation laws}
\label{ss:2.3}
Let $U$ denote the compact Polish space of measurable functions $u:\bb
T \to [0,1]$, equipped with the metric it inherits as a (closed)
subset of $H^{-1}(\bb T)$, namely
\begin{eqnarray*}
 d_U (u,v) := \sup \Big\{ \langle u-v,\varphi \rangle,\, \varphi \in H^1(\bb
T)\,:\:
 \|\varphi\|_{L_2(\bb T)}^2+
  \|\nabla \varphi\|_{L_2(\bb T)}^2\le 1 \Big\}
\end{eqnarray*}

Fix $T>0$ and consider the formal limiting equation for \eqref{e:2.2}
\begin{eqnarray}
\label{e:2.4}
\nonumber
& & 
\partial_t u + \nabla \cdot f (u) =0
\\ & & 
u(0,x) = u_0(x)
\end{eqnarray}
In general there exist no smooth solutions to \eqref{e:2.4}. A function
$u\in C\big([0,T]; U\big)$ is a \emph{weak solution} to \eqref{e:2.4} iff
for each $\varphi\in C^\infty([0,T]\times \bb T)$ it satisfies
\begin{eqnarray*}
\langle u(T),\varphi(T)\rangle -\langle u_0,\varphi(0)\rangle
-\langle\langle u, \partial_t\varphi\rangle\rangle 
-\langle\langle f(u),\nabla \varphi\rangle\rangle = 0
\end{eqnarray*}
As well known \cite[Chap.~6]{Daf}, existence and uniqueness of a weak
\emph{Kruzkov} solution to \eqref{e:2.4} is guaranteed under an
additional entropic condition, which is recalled in
Section~\ref{ss:2.5} below. Then $u^\eps$ converges in probability to
such a solution both in the strong $L_p([0,T]\times \bb T)$
and $C\big([0,T]; U\big)$ topologies.
\begin{proposition}
\label{t:uconv}
Assume $\lim_\eps \eps^{2(\gamma-1)} \big[ \|\jmath^\eps\|_{L_2(\bb
  T)}^2 + \eps \|\nabla \jmath^\eps\|_{L_2(\bb T)}^2 \big]=0$. Let
$\bar{u}$ be the unique Kruzkov solution to \eqref{e:2.4}. Then for each
$p<+\infty$ and $\delta>0$
\begin{eqnarray*}
\lim_\eps P \big( \|u^\eps-\bar{u}\|_{L_p([0,T]\times \bb T)}^p
+  \sup_{t\in [0,T]}d_{U}(u^\eps(t),\bar{u}(t)) > \delta \big)=0
\end{eqnarray*}
\end{proposition}
Proposition~\ref{t:uconv} establishes a convergence result for the
probability law of the process $u^\eps$ solution to \eqref{e:2.2}, as
$\eps \to 0$. We are then interested in large deviations principles
for this probability law. We recall the definition of the large
deviations bounds \cite{DZ2}.

\begin{definition}
Let $\mc X$ be a Polish space and $\{\bb P^\eps\} \subset \mc P(\mc X)$ a
family of Borel probability measures on $\mc X$. For $\{\alpha_\eps\}$ a
sequence of positive reals such that $\lim_\eps \alpha_\eps=0$ and $I:\mc X
\to [0,+\infty]$ a lower semicontinuous functional, we say that $\{\bb
P^\eps\}$ satisfies
\begin{itemize}
 \item[-] {A \emph{large deviations upper bound} with speed $\alpha_\eps^{-1}$
and
rate $I$, iff for each closed set $\mc C \subset \mc X$
\begin{eqnarray}
 \varlimsup_\eps \alpha_\eps \log \bb P^\eps(\mc C) \le - \inf_{u \in \mc C}
I(u)
\end{eqnarray}
}

 \item[-] {A \emph{large deviations lower bound} with speed $\alpha_\eps^{-1}$
and
rate $I$, iff for each open set $\mc O \subset \mc X$
\begin{eqnarray}
 \varlimsup_\eps \alpha_\eps \log \bb P^\eps(\mc O) \ge - \inf_{u \in \mc O}
I(u)
\end{eqnarray}
}
\end{itemize}
$\{\bb P^\eps\}$ is said to satisfy a \emph{large deviations
  principle} if both the upper and lower bounds hold with same rate
and speed.
\end{definition}
In the next sections, we introduce some preliminary notions and state
a first large deviations principle with speed $\eps^{-2\gamma}$. We
next introduce some additional preliminaries and state a second large
deviations partial result, associated with the speed
$\eps^{-2\gamma+1}$ .
\subsection{First order large deviations}
\label{ss:2.4}
We first introduce a suitable space $\mc M$ of Young measures and
recall the notion of measure-valued solution to \eqref{e:2.4}.
Consider the set $\mc N$ of measurable maps $\mu$ from $[0,T]\times
\bb T$ to the set $\mc P([0,1])$ of Borel probability measures on
$[0,1]$. The set $\mc N$ can be identified with the set of positive
finite Borel measures $\mu$ on $[0,T]\times \bb T \times [0,1]$ such
that $\mu(dt,\, dx,\,[0,1])=dt\,dx$, by the bijection
$\mu(dt,\,dx,\,d\lambda)=dt\,dx\,\mu_{t,x}(d\lambda)$. For $\imath :
[0,1]\to [0,1]$ the identity map, we set
\begin{eqnarray*}
 \mc M := \big\{\mu \in \mc N\,:\:\text{the map
   $[0,T]\ni t \mapsto \mu_{t,\cdot}(\imath)$ is in 
             $C\big([0,T]; U\big)$} \big\}
\end{eqnarray*}
in which, for a bounded measurable function $F:[0,1]\to \bb R$, the
notation $\mu_{t,x}(F)$ stands for $\int_{[0,1]} \!
\mu_{t,x}(d\lambda) F(\lambda)$. We endow $\mc M$ with the metric
\begin{eqnarray*}
  d_{\mc M} (\mu,\nu) := d_{\mathrm{*w}}(\mu,\nu) 
  +\sup_{t\in [0,T]} 
        d_{U}\big(\mu_{t,\cdot}(\imath),\nu_{t,\cdot}(\imath)\big)
\end{eqnarray*}
where $d_{\mathrm{*w}}$ is a distance generating the relative topology
on $\mc N$ regarded as a subset of the finite Borel measures on
$[0,T]\times \bb T \times [0,1]$ equipped with the $*$-weak topology.
$(\mc M, d_{\mc M})$ is a Polish space.

An element $\mu\in \mc M$ is a \emph{measure-valued solution} to
\eqref{e:2.4} iff for each $\varphi\in C^\infty \big([0,T]\times \bb
T)$ it satisfies
\begin{eqnarray*}
\langle \mu_{T,\cdot}(\imath),\varphi(T)\rangle
-\langle u_0,\varphi(0)\rangle
-  \langle\langle \mu(\imath), \partial_t \varphi \rangle\rangle
- \langle\langle \mu(f), \nabla\varphi \rangle\rangle =0
\end{eqnarray*}
If $u\in C\big([0,T]; U\big)$ is a weak solution to \eqref{e:2.4},
then the map $(t,x) \mapsto \delta_{u(t,x)}(d\lambda)\in \mc P([0,1])$
is a measure-valued solution. However, in general there exist
measure-valued solutions which do not have this form, namely they are
not a Dirac mass at a.e.\ $(t,x)$ (e.g.\ finite convex combinations of
Dirac masses centered on weak solutions are measure-valued solutions).

Consider the process $\mu^\eps:\Omega \to \mc M$ defined by
$\mu^\eps_{t,x}:=\delta_{u^\eps(t,x)}$. We let $\mb P^\eps:=P \circ
(\mu^\eps)^{-1} \in \mc P(\mc M)$ be the law of $\mu^\eps$ on $\mc
M$. In Section~\ref{ss:3.2} we prove
\begin{theorem}
\label{t:ld1}
Assume $\lim_\eps \eps^{2(\gamma-1)}\big[\|\jmath^\eps\|_{L_2}^2 +
\eps\|\nabla\jmath^\eps\|_{L_2}^2\big]=0$.
\begin{itemize}
\item[\rm{(i)}]{Then the sequence $\{\mb P^\eps\}\subset \mc P(\mc M)$
    satisfies a large deviations upper bound on $\mc M$ with speed
    $\eps^{-2\gamma}$ and rate functional $\mc I:\mc M \to
    [0,+\infty]$ defined as
\begin{eqnarray}
  \label{e:2.5}
\nonumber
& &  \mc I (\mu) :=\sup_{\varphi \in C^\infty([0,T] \times \bb T) } \Big\{ 
  \langle \mu_{T,\cdot}(\imath),\varphi(T) \rangle 
  -\langle u_0,\varphi(0)\rangle
  - \langle\langle \mu(\imath), \partial_t \varphi \rangle\rangle 
\\ & & \phantom{ 
\mc I (\mu) :=\sup_{\varphi \in C^\infty([0,T] \times \bb T) } \Big\{
}
  - \langle\langle \mu (f), \nabla \varphi \rangle\rangle 
  - \frac 12\, \langle\langle \mu (a^2) \nabla \varphi, 
     \nabla \varphi \rangle\rangle 
  \Big\}
\end{eqnarray}
}

\item[\rm{(ii)}]{Assume furthermore that $\zeta \le
    u_0\le 1-\zeta$ for some $\zeta>0$. Then $\{\mb P^\eps\} \subset
    \mc P(\mc M)$ satisfies a large deviations lower bound on $\mc M$
    with speed $\eps^{-2\gamma}$ and rate functional $\mc I$.  }
\end{itemize}
\end{theorem}
We denote by $\bb P^\eps:=P \circ (u^\eps)^{-1} \in \mc
P\big(C\big([0,T]; U\big)\big)$ the law of $u^\eps$ on the Polish
space $(C\big([0,T]; U\big)$. By contraction principle
\cite[Theorem~4.2.1]{DZ2} we get
\begin{corollary}
\label{c:ld1}
Under the same hypotheses of Theorem~\ref{t:ld1}, the sequence $\{\bb
P^\eps\} \subset \mc P\big(C\big([0,T]; U\big)\big)$ satisfies a large
deviations principle on $C\big([0,T]; U\big)$ with speed
$\eps^{-2\gamma}$ and rate functional $I:C\big([0,T]; U\big) \to
[0,+\infty]$ defined as
\begin{eqnarray*}
& & 
 I(u)  := 
  \inf \Big\{  
  \int \!dt \, dx\, 
  R_{f,a^2}\big(u(t,x),\Phi(t,x)\big),
\\ & & \phantom{I(u)  := \inf \Big\{ }
   \Phi\in L_2([0,T]\times \bb T) \,:\: 
    \nabla \Phi=- \partial_t u \text{ weakly}
  \Big\}
\end{eqnarray*}
where $R_{f,a^2}:[0,1]\times \bb R \to [0,+\infty]$ is defined by
\begin{eqnarray*}
  R_{f,a^2}(w,c):=\inf \{\big(\nu(f)-c\big)^2 / \nu(a^2),
            \,\nu   \in \mc P([0,1])\,:\: \nu(\imath)=w \}
\end{eqnarray*}
in which we understand $(c-c)^2/0=0$.
\end{corollary}
Note that, if $\mc I(\mu)<+\infty$, then $\mu_{0,x}(\imath)=u_0(x)$
and analogously $I(u)<+\infty$ implies $u(0,x)=u_0(x)$. On the other
hand, $\mc I(\mu)=0$ iff $\mu$ is a measure-valued solution to
\eqref{e:2.4}. $\mc I(\mu)$ quantifies indeed how $\mu$ deviates from
being a measure-valued solution to \eqref{e:2.4} in a suitable Hilbert
norm, see the proof of Theorem~\ref{t:ld1} item (i) in
Section~\ref{ss:3.2}. On the other hand, if $f$ is nonlinear, in
general we have $I(u)<\mc I(\delta_u)$, so that $I$ vanishes on a set
wider than the set of weak solutions to \eqref{e:2.4}.

In general there exist infinitely many measure-valued solutions to
\eqref{e:2.4}, but Proposition~\ref{t:uconv} implies that $\{\mb
P^\eps\}$ converges in probability in $\mc M$ to the unique Kruzkov
solution $\bar{u}$ to \eqref{e:2.4} (more precisely, to the Young
measure $\bar{\mu}$ defined by
$\bar{\mu}_{t,x}=\delta_{\bar{u}(t,x)}$). We thus expect that
additional nontrivial large deviations principles may hold with a
speed slower than $\eps^{-2\gamma}$.

\subsection{Entropy-measure solutions to conservation laws}
\label{ss:2.5}
Let $\mc X$ be the same set $C([0,T];U)$ endowed with the
metric
\begin{eqnarray*}
  d_{\mc X}(u,v):=  \|u-v\|_{L_1([0,T]\times \bb T)}
          + \sup_{t\in [0,T]} d_{U}(u(t), v(t))
\end{eqnarray*}
Convergence in $\mc X$ is of course strictly stronger than convergence
in $C\big([0,T]; U \big)$, since convergence in $L_p([0,T]\times \bb
T)$ for $p\in [1,+\infty)$ is also required. Note that $\mc X$ can be
identified with the subset of $\mc M$
\begin{equation*}
\big\{ \mu \in \mc M \,:\:\mu=\delta_u,\,\text{for some $u \in
 C\big([0,T]; U\big)$} \big\}
\end{equation*}
and $d_{\mc X}$ is indeed a distance generating the relative topology
induced by $d_{\mc M}$ on $\mc X$. In particular, once exponential
tightness is established on $\mc X$, it is immediate to \emph{lift}
large deviations principles for the law of $u^\eps$ on $\mc X$, to the
corresponding law of $\delta_{u^\eps}$ on $\mc M$.

A function $\eta \in C^2([0,1])$ is called an \emph{entropy} and its
\emph{conjugated entropy flux} $q\in C([0,1])$ is defined up to a
constant by $q(u):=\int^u\!dv\,\eta'(v)f'(v)$. For $u$ a weak solution
to \eqref{e:2.4}, for $(\eta,q)$ an entropy--entropy flux pair, the
\emph{$\eta$-entropy production} is the distribution $\wp_{\eta,u}$
acting on $C^\infty_{\mathrm{c}}\big([0,T)\times \bb T\big)$ as
\begin{eqnarray}
  \label{e:2.6}
\wp_{\eta,u}(\varphi):=  -\langle \eta(u_0),\varphi(0)\rangle
    - \langle\langle \eta(u) ,\partial_t \varphi \rangle\rangle
    - \langle\langle q(u) , \nabla \varphi \rangle\rangle
\end{eqnarray}

Let $C^{2,\infty}_{\mathrm{c}}\big([0,1]\times [0,T)\times \bb T
\big)$ be the set of compactly supported maps $\vartheta:[0,1]\times
[0,T)\times \bb T \ni (v,t,x) \to \vartheta(v,t,x) \in \bb R$, that
are twice differentiable in the $v$ variable, with derivatives
continuous up to the boundary of $[0,1]\times [0,T)\times \bb T$, and
that are infinitely differentiable in the $(t,x)$ variables. For
$\vartheta \in C^{2,\infty}_{\mathrm{c}}\big([0,1]\times [0,T)\times
\bb T \big)$ we denote by $\vartheta'$ and $\vartheta''$ its partial
derivatives with respect to the $v$ variable. We say that a function
$\vartheta \in C^{2,\infty}_{\mathrm{c}}\big([0,1]\times [0,T)\times
\bb T \big)$ is an \emph{entropy sampler}, and its \emph{conjugated
  entropy flux sampler} $Q:[0,1]\times [0,T)\times \bb T$ is defined
up to an additive function of $(t,x)$ by
$Q(u,t,x):=\int^u\!dv\,\vartheta'(v,t,x) f'(v)$. Finally, given a weak
solution $u$ to \eqref{e:2.4}, the \emph{$\vartheta$-sampled entropy
  production} $P_{\vartheta,u}$ is the real number
\begin{eqnarray}
\label{e:2.6b}
\nonumber
P_{\vartheta,u}
&:=& -\int\!dx\,\vartheta(u_0(x),0,x)
\\ & & -\int \! dt\,dx\,
  \Big[\big(\partial_t \vartheta)\big(u(t,x),t,x \big) 
       + \big(\partial_x Q\big)\big(u(t,x),t,x \big)\Big]
\end{eqnarray}
If $\vartheta(v,t,x)=\eta(v) \varphi(t,x)$ for some entropy $\eta$ and
some $\varphi \in C^{\infty}_{\mathrm{c}}\big([0,T)\times \bb T\big)$,
then $P_{\vartheta,u}=\wp_{\eta,u}(\varphi)$.

We next introduce a suitable class of solutions to \eqref{e:2.4} for
later use. We denote by $M\big([0,T)\times \bb T\big)$ the set of
Radon measures on $[0,T) \times \bb T$ that we consider equipped with
the vague topology. In the following, for $\wp \in M\big([0,T)\times
\bb T\big)$ we denote by $\wp^\pm$ the positive and negative part of
$\wp$. For $u$ a weak solution to \eqref{e:2.4} and $\eta$ an entropy,
recalling \eqref{e:2.6} we set
\begin{eqnarray*}
\|\wp_{\eta,u}\|_{\mathrm{TV}}:=
    \sup \big\{\wp_{\eta,u}(\varphi),\,\varphi \in
               C^\infty_{\mathrm{c}}\big([0,T)\times \bb T \big),\,
               |\varphi|\le 1
         \big\} 
\end{eqnarray*}
\begin{eqnarray*}
\|\wp_{\eta,u}^+\|_{\mathrm{TV}}:= 
    \sup \big\{\wp_{\eta,u}(\varphi),\,\varphi \in
               C^\infty_{\mathrm{c}}\big([0,T)\times \bb T \big),\,
               0 \le \varphi \le 1
         \big\}
\end{eqnarray*}

The following result follows by adapting \cite[Prop.~2.3]{BBMN} and
\cite[Prop.~3.1]{DOW} to the setting of this paper.
\begin{proposition}
\label{p:kin}
Let $u \in \mc X$ be a weak solution to \eqref{e:2.4}. The following
statements are equivalent:
\begin{itemize}
\item[{\rm (i)}]{ For each entropy $\eta$, the $\eta$-entropy
    production $\wp_{\eta,u}$ can be extended to a Radon measure on
    $[0,T)\times \bb T$, namely
    $\|\wp_{\eta,u}\|_{\mathrm{TV}}<+\infty$ for each entropy $\eta$.}

\item[{\rm (ii)}]{There exists a bounded measurable map
    $\varrho_u:[0,1] \ni v \to \varrho_u(v;dt,dx) \in
    M\big([0,T)\times \bb T\big)$ such that for any entropy sampler
    $\vartheta$
\begin{eqnarray*}
P_{\vartheta,u} = \int \! dv\,\varrho_u(v;dt,dx)\,\vartheta''(v,t,x)
\end{eqnarray*}
}
\end{itemize}
\end{proposition}
A weak solution $u \in \mc X$ that satisfies the equivalent conditions
in Proposition~\ref{p:kin} is called an \emph{entropy-measure
  solution} to \eqref{e:2.4}. We denote by $\mc E \subset \mc X$ the
set of entropy-measure solutions to \eqref{e:2.4}.

A weak solution $u\in \mc X$ to \eqref{e:2.4} is called an
\emph{entropic solution} iff for each convex entropy $\eta$ the
inequality $\wp_{\eta,u} \le 0$ holds in distribution sense, namely
$\|\wp_{\eta,u}^+\|_{\mathrm{TV}}=0$. Entropic solutions are
entropy-measure solutions such that $\varrho_u(v;dt,dx)$ is a negative
Radon measure for each $v \in [0,1]$. It is well known, see e.g.\
\cite[Theorem~6.2.1]{Daf}, that for each $u_0 \in U$ there exists
a unique entropic weak solution $\bar{u} \in \mc X \cap
C\big([0,T];L_1(\bb T)\big)$ to \eqref{e:2.4}. Such a solution is
called the \emph{Kruzkov solution} with initial datum $u_0$.

Up to minor adaptations, the following class of solutions have been
also introduced in \cite{BBMN}, where some examples of such solutions
are also given.
\begin{definition}
\label{d:splittable}
An entropy-measure solution $u \in \mc E$ is \emph{entropy-splittable}
iff there exist two closed sets $E^+,E^- \subset [0,T]\times \bb T$
such that
\begin{itemize}
\item[{\rm (i)}] {For a.e.\ $v \in [0,1]$, the support of
  $\varrho_u^+(v;dt,dx)$ is contained in $E^+$, and the support of
  $\varrho_u^-(v;dt,dx)$ is contained in $E^-$.}
\item[{\rm (ii)}] {The set $\big\{ t \in [0,T]\,:\:
    \big(\{t\} \times \bb T \big) \cap E^+ \cap E^- \neq \emptyset
    \big\}$ is nowhere dense in $[0,T]$.}
\item[{\rm (iii)}] {There exists $\delta>0$ such that $\delta\le u\le
    1-\delta$. }
\end{itemize}
The set of entropy-splittable solutions to \eqref{e:2.4} is denoted by
$\mc S$.
\end{definition}
Note that $\mc S \subset \mc E \subset \mc X$, and if $u_0$ is bounded
away from $0,\,1$, then $\mc S$ is nonempty (for instance the Kruzkov
solution to \eqref{e:1.3} is in $\mc S$). Indeed in general $\mc S
\not\subset BV\big([0,T]\times \bb T \big)$.

\subsection{Second order large deviations}
\label{ss:2.6}
With a little abuse of notation, we still denote with $\bb P^\eps:=P
\circ (u^\eps)^{-1} \in \mc P(\mc X)$ the law of $u^\eps$ on the
Polish space $(\mc X,d_{\mc X})$. Since $\int \! dx\,
\jmath^\eps(x)=1$ (see hypothesis \textbf{H4)}), we have that
$\jmath^\eps-1$ is the derivative of some smooth function $J$ on $\bb
T$, defined up to an additive constant. We define $\|\jmath^\eps
-\id\|_{W^{-1,1}(\bb T)}$ as the infimum of $\|J\|_{L_1(\bb T)}$ as
$J$ runs on the set of functions $J$ such that $\nabla \cdot
J=\jmath^\eps-1$. We have the following
\begin{theorem}
\label{t:ld2}
Assume that there is no interval in $[0,1]$ where $f$ is affine, and
that $\lim_\eps \eps^{2(\gamma-1)} \big[ \|\jmath^\eps\|_{L_2(\bb
  T)}^2 + \eps \|\nabla \jmath^\eps\|_{L_2(\bb T)}^2 \big]=0$.
\begin{itemize}
\item[(i)]{Then the sequence $\{\bb P^\eps\} \subset \mc P(X)$
    satisfies a large deviations upper bound on $(\mc X,d_{\mc X})$
    with speed $\eps^{-2\gamma+1}$ and rate functional $H:\mc X \to
    [0,+\infty]$ defined as
\begin{eqnarray*}
  H(u):=
\begin{cases}
  \displaystyle{ \int \! dv\,\varrho_u^+(v;dt,dx)\,
    \frac{D(v)}{a^2(v)} } & \text{if $u\in \mc E$ }
  \\
  + \infty & \text{ otherwise }
\end{cases}
\end{eqnarray*}
}

\item[(ii)]{Assume furthermore $\lim_\eps \eps^{-3/2}
    \|\jmath^\eps-\id\|_{W^{-1,1}(\bb T)}=0$ and $f \in
    C^2([0,1])$. Then the sequence $\{\bb P^\eps\} \subset \mc P(X)$
    satisfies a large deviations lower bound on $(\mc X,d_{\mc X})$
    with speed $\eps^{-2\gamma+1}$ and rate functional $\upbar{H}:\mc
    X \to [0,+\infty]$ defined as
\begin{eqnarray*}
  \upbar{H}(u) := \sup_{\substack{\mc O \ni u \\ \mc O\,\mathrm{open} }}
  \:\inf_{v\in \mc O \cap \mc S} \, H(v)
\end{eqnarray*}
}
\end{itemize}
\end{theorem}
Since $H$ is lower semicontinuous on $\mc X$, we have $\upbar{H} \ge
H$ on $\mc X$ and $\upbar{H}= H$ on $\mc S$, namely a large deviations
principle holds on $\mc S$. In order to obtain a full large deviations
principle, one needs to show $H(u)\ge \upbar{H}(u)$ for $u \not\in \mc
S$. This amounts to show that $\mc S$ is $H$-dense in $\mc X$, namely
that for $u \in \mc X$ such that $H(u)<+\infty$ there exists a
sequence $\{u^n\} \subset \mc S$ converging to $u$ in $\mc X$ such
that $H(u^n)\to H(u)$. In particular it can be shown that
$\upbar{H}(u)=H(u)$ for $u$ piecewise smooth. The main difficulties
here arise from the lacking of a chain rule formula connecting the
measures $\varrho_u$ to the structure of $u$ itself. If $u$ has
bounded variation, Vol'pert chain rule \cite{AFP} allows an explicit
representation for $\varrho_u$ and thus $H(u)$, see Remark~2.7 in
\cite{BBMN}. On the other hand, there exists $u \in \mc X$ with
infinite variation such that $H(u)<+\infty$, see Example~2.8 in
\cite{BBMN}. While chain rule formulas out of the BV setting are
subject to current research investigation, see e.g.\ \cite{DOW,ADM},
only partial results are available.

Under the same hypotheses of Theorem~\ref{t:ld2}, one can show that
entropy-measure solutions to \eqref{e:2.4} are in $C([0,T];L_1(\bb
T))$, see Lemma~5.1 in \cite{BBMN}. By Kruzkov uniqueness theorem
\cite[Theorem~6.2.1]{Daf}, we gather that $H(u)=0$ iff $u$ is the
Kruzkov solution to \eqref{e:2.4} with initial datum $u_0$. In
particular, by item (i) in Theorem~\ref{t:ld2}, large deviations
principles with speeds slower than $\eps^{-2\gamma+1}$ are trivial.

Note that in Proposition~\ref{t:uunique}, Proposition~\ref{t:uconv},
Theorem~\ref{t:ld1} and Theorem~\ref{t:ld2} various hypotheses on
$\jmath^\eps$ are required, the most restrictive in
Theorem~\ref{t:ld2}. It is easy to see that, if $\gamma >1$, there
exist convolution kernels $\jmath^\eps$ satisfying them all.

\section{Proofs}
\label{s:3}
\subsection{Convergence and bounds}
\label{ss:3.1}
In the following we will need to consider several different
perturbations of \eqref{e:2.2}. In the next lemma we write down
explicitly an It\^{o} formula for \eqref{e:2.2}. The corresponding
It\^{o} formula for the perturbed equations can be obtained
analogously, as the martingale term in these equations is always the
same.
\begin{lemma}[It\^{o} formula]
\label{l:ito}
Let $(\vartheta; Q)$ be an entropy sampler--entropy sampler flux pair
for the equation \eqref{e:2.4} (recall in particular
$\vartheta(u,T,x)=0$). Then
\begin{eqnarray}
\label{e:ito1}
\nonumber
  & & 
  -\int\!dx\, \vartheta(u_0(x),0,x)
  -\int \! dt\,dx\,
  \big[\big(\partial_t \vartheta)\big(u^\eps(t,x),t,x \big) 
  + \big(\partial_x Q\big)\big(u^\eps(t,x),t,x \big)\big]
  \\  \nonumber & &
\qquad \qquad =
 -\frac{\eps}2 \langle \langle  \vartheta''(u^\eps) \nabla u^\eps,
  D(u^\eps) \nabla u^\eps \rangle \rangle
  -\frac{\eps}2 \langle \langle \partial_x \vartheta'(u^\eps),
  D(u^\eps) \nabla u^\eps\rangle \rangle
  \\ \nonumber 
& & \qquad \qquad \phantom{=}
   + \frac{\eps^{2\gamma}}{2} \|\nabla \jmath^\eps \|_{L_2(\bb T)}^2
    \langle \langle \vartheta''(u^\eps) a(u^\eps), a(u^\eps)\rangle \rangle
\\ & & 
\qquad \qquad \phantom{=}
  + \frac{\eps^{2\gamma}}{2}  \|\jmath^\eps \|_{L_2(\bb T)}^2
       \langle \langle \vartheta''(u^\eps) \nabla u^\eps, 
              [a'(u^\eps)]^2 \nabla u^\eps\rangle \rangle
 +N^{\eps;\vartheta}(T)
\end{eqnarray}
where $N^{\eps;\vartheta}$ is the martingale
\begin{eqnarray}
\label{e:Nvartheta}
  N^{\eps;\vartheta}(t)
:=   -\eps^\gamma \int_0^t \big\langle \jmath^\eps \ast
     \big[a(u^\eps) \vartheta''(u^\eps)\nabla u^\eps +
          a(u^\eps) \partial_x \vartheta'(u^\eps)\big], dW \big\rangle
\end{eqnarray}
Moreover the quadratic variation of $N^{\eps,\vartheta}$ enjoys the bound
\begin{equation}
\label{e:youngquad}
\big[ N^{\eps;\vartheta},  N^{\eps;\vartheta} \big](t)\le  
\eps^{2\gamma}
\big\| a(u^\eps) \big[\vartheta''(u^\eps) \nabla u^\eps 
   +\partial_x \vartheta'(u^\eps) \big] \big\|_{L_2([0,t]\times \bb T)}^2
\end{equation}
\end{lemma}
\begin{proof}
  Equation \eqref{e:ito1} follows, up to minor manipulations, from
  It\^{o} formula \cite[Theorem~4.17]{DZ} for the map
\begin{equation*}
[0,T]\times U \ni (t,u) \mapsto \int\!dx\,\vartheta(u(x),t,x) \in \bb R
\end{equation*}
By \eqref{e:Nvartheta} and \eqref{e:2.1}, the quadratic variation of
$N^{\eps;\vartheta}$ is given by
\begin{equation*}
\big[ N^{\eps;\vartheta},  N^{\eps;\vartheta} \big](t)
= \eps^{2\gamma}  
\big\|\jmath^\eps \ast \big\{ 
    a(u^\eps) \big[\vartheta''(u^\eps) \nabla u^\eps 
   +\partial_x \vartheta'(u^\eps) \big] \big\}\big\|_{L_2([0,t]\times \bb T)}^2
\end{equation*}
so that the inequality stated in the lemma follows by Young inequality
for convolutions and hypothesis \textbf{H4)}.  
\end{proof}

\begin{lemma}
\label{l:dismart}
Let $\zeta,\,T>0$, let $X$ be a real, continuous, local, square
integrable supermartingale starting from $0$, and let $\tau \le T$ be
a stopping time. Let $F: \bb R \to \bb R^+$ be such that:
\begin{eqnarray}
\label{e:Fmart}
\frac{F(x)}{F(\zeta)} \le 2\frac{x}{\zeta}-1, 
    \quad \text{for all $x >\zeta$}
\end{eqnarray}
Then:
\begin{eqnarray}
\label{e:bernstein.a}
\bb P \Big(\sup_{0\le t \le \tau} X(t) \ge \zeta,
    \, \big[X, X \big](\tau) \le F(\sup_{t\le \tau} X(t)) \Big) 
\le
\exp\Big[-\frac{\zeta^2}{2 F(\zeta)}\Big]
\end{eqnarray}
\end{lemma}
Note that the hypotheses \eqref{e:Fmart} on $F$ are satisfied by any
nonincreasing function, and by functions with affine or subaffine
behavior. Lemma~\ref{l:dismart} provides an elementary generalization
of the well known Bernstein inequality \cite[page~153]{RY}, which
deals with the case of constant $F$.
\begin{proof}
  Hypotheses on $F$ imply that the map $G_\zeta: x \to
  \frac{\zeta}{F(\zeta)}x -\frac{1}{2}\frac{\zeta^2}{F(\zeta)^2}F(x)$
  satisfies $G_\zeta(x)\ge G_\zeta(\zeta)=\frac{\zeta^2}{2 F(\zeta)}$
  for all $x \ge \zeta$. Therefore:
\begin{eqnarray*}
& & \bb P \Big(\sup_{t\le \tau} X(t) \ge \zeta, 
\, \big[X, X \big](\tau) \le F(\sup_{t\le \tau} X(t))\Big)
\\ & & 
\le \bb P \Big(e^{\frac{\zeta}{F(\zeta)} 
\sup_{t\le \tau} X(t) -\frac{1}{2}\frac{\zeta^2}{F(\zeta)^2}
F(\sup_{t\le \tau} X(t))} \ge e^{\frac{1}{2}\frac{\zeta^2}{F(\zeta)}},
\\ & & 
\phantom{ \le \bb P \Big(  
 }
 \big[X, X \big](\tau) \le F(\sup_{t\le \tau} X(t)) \Big)
\\ & & 
\le  \bb P \Big(\sup_{t\le T} 
e^{\frac{\zeta}{F(\zeta)}X(t) -\frac{1}{2}
\frac{\zeta^2}{F(\zeta)^2}[X,X](t)} 
\ge e^{\frac{1}{2}\frac{\zeta^2}{F(\zeta)}}\Big) 
\le e^{-\frac{\zeta^2}{2 F(\zeta)}}.
\end{eqnarray*}
where in the last line we applied the maximal inequality for positive
supermartingales \cite[page~58]{RY}, to the supermartingale
$e^{\frac{\zeta}{F(\zeta)}X(t) -\frac{1}{2}
  \frac{\zeta^2}{F(\zeta)^2}[X,X](t)}$.
  \end{proof}

The next lemma provides a key a priori bound.
\begin{lemma}
\label{l:supexpnabla}
For $\eps>0$, let $E^\eps \in L_2\big([0,T];H^1(\bb T)\big)$ and let
$\bb Q^\eps \in \mc P\big(C\big([0,T]; U\big)\big)$ be any martingale
solution to the Cauchy problem
\begin{eqnarray}
\nonumber
& & d u =  \big[- \nabla \cdot f(u)
+\frac{\eps}{2} \nabla \cdot \big(D(u) \nabla
u \big) - \nabla \cdot( a(u) E^\eps)\big]\,dt
\\ \nonumber
& & \phantom{ d u =}
 +\eps^{\gamma}\, \nabla \cdot
\big[a(u) (\jmath^\eps \ast d W)\big]
\\ 
& & u(0,x)  = u_0^\eps(x)
\label{e:stoc2b}
\end{eqnarray}
Assume $\|\nabla E^\eps\|_{L_2([0,T]\times \bb T)} \le C_0$ for some
constant $C_0$ independent of $\eps$, and $\lim_\eps
\eps^{2\gamma-1}(\|\jmath^\eps\|_{L_2(\bb T)}^2+ \eps\| \nabla
\jmath^\eps\|_{L_2(\bb T)}^2)=0$. Then there exist $C,\, \eps_0>0$
such that for any $\eps <\eps_0$:
\begin{equation}
\label{e:superexp.a}
\eps \| \nabla u\|_{L_2([0,T] \times \bb T)}^2
\le
C+N^\eps(T,u) \qquad \text{for $\bb Q^\eps$ a.e.\ $u$ }
\end{equation}
where $N^\eps$ is a $\bb Q^\eps$-martingale starting from $0$ and
satisfying
\begin{equation}
\label{e:superexp.d}
\bb Q^\eps \big(\sup_{t \le T} N^\eps(t) >\zeta \big)
\le \exp \Big\{- \frac{\zeta^2}{\eps^{2\gamma-1}C (1+\zeta)}\Big\}
\end{equation}
\end{lemma}
\begin{proof}
  It\^{o} formula for the map $U \ni u \mapsto \int\!dx\, u^2(x) \in
  \bb R$ can be obtained as in Lemma~\ref{l:ito}, so that
\begin{eqnarray*}
\nonumber
  & &
\|u(T)\|_{L_2(\bb T)} - \|u_0\|_{L_2(\bb T)}
 +\eps \langle \langle  \nabla u,
  D(u) \nabla u \rangle \rangle
\\ & & 
\quad = - \langle \langle A(u),\nabla E^\eps \rangle \rangle
+  \eps^{2\gamma} \|\nabla \jmath^\eps\|_{L_2(\bb T)}^2\,
  \| a(u) \|_{L_2([0,T] \times \bb T)}^2
\\ & & \phantom{\quad =} 
    + \eps^{2\gamma} \|\jmath^\eps\|_{L_2(\bb T)}^2\,
  \| a'(u) \nabla u\|_{L_2([0,T] \times \bb T)}^2
+N^{\eps}(T,u)
\end{eqnarray*}
where $A \in C^1([0,1])$ is any antiderivative of $a(\cdot)$ and
$N^\eps$ is a $\bb Q^\eps$-martingale, which -reasoning as in the
proof of \eqref{e:youngquad}- satisfies
\begin{eqnarray}
\label{e:quadvar3}
\big[N^{\eps}, N^{\eps}\big](T,u)
 \le  4\,\eps^{2\gamma} \| a(u) \nabla u \|_{L_2([0,T] \times \bb T)}^2
\end{eqnarray}
By \textbf{H2)}, \textbf{H3)} and the hypotheses of this lemma, there
exist $C_1,\,\eps_0>0$ such that, for each $\eps \le \eps_0$ and $v
\in [0,1]$
\begin{equation*}
\eps^{2\gamma} \|\jmath^\eps\|_{L_2(\bb T)}^2 [a'(v)]^2
\le \frac{\eps}2 D(v)
\end{equation*}
\begin{equation*}
|\langle \langle A(u),\nabla E^\eps
\rangle \rangle|+\eps^{2\gamma} \|\nabla \jmath^\eps\|_{L_2(\bb T)}^2\,
\|a(u)\|_{L_2([0,T] \times \bb T)}^2 \le C_1
\end{equation*}
Therefore, since $|u_0|\le 1$
\begin{eqnarray}
\label{e:parest}
  \frac{\eps}2 \langle \langle  \nabla u,
  D(u) \nabla u \rangle \rangle \le 1+C_1 
  + N^{\eps}(T,u)
\end{eqnarray}
and thus \eqref{e:superexp.a} since $D$ is uniformly positive. By
\eqref{e:quadvar3} and \eqref{e:parest}, there exists a constant
$C_2>0$ such that
\begin{equation}
\big[N^{\eps}, N^{\eps}\big](T,u) \le 
 \eps^{2\gamma}C_2 \langle \langle  \nabla u,
  D(u) \nabla u \rangle \rangle \le 2\,C_2 \,\eps^{2\gamma-1}
           \big[1+C_1+ N^{\eps}(T) \big]
\end{equation}
This inequality allows the application of Lemma~\ref{l:dismart} for
the martingale $N^{\eps}$ with
\begin{equation*}
F(\zeta)=2\,C_2 \,\eps^{2\gamma-1} (1+C_1+ \zeta)
\end{equation*}
which clearly satisfies the condition \eqref{e:Fmart}. The bound
\eqref{e:superexp.d} then follows straightforwardly.  
\end{proof}

The following lemma provides a stability result for \eqref{e:2.2}. It
will be repeatedly used to evaluate the effects of the Girsanov terms
appearing in \eqref{e:2.2} when absolutely continuous perturbations of
$\bb P^\eps$ are considered.
\begin{lemma}
\label{l:stab}
For each $\eps>0$, let $v^\eps:\mc X \to \mc X \cap L_2([0,T];H^1(\bb
T))$ and $G^\eps:\mc X \times \mc X \to L_2([0,T]\times \bb T)$ be
adapted maps (with respect to the standard filtrations of $\mc X$ and
$\mc X \times \mc X$ respectively). Let $\bb Q^\eps \in \mc P(\mc X)$
be any martingale solution to the stochastic Cauchy problem in the
unknown $u$
\begin{eqnarray}
\label{e:stocgir}
\nonumber
& & d u  =  \big[- \nabla \cdot f(u)
+\frac{\eps}{2} \nabla \cdot \big(D(u) \nabla
u \big)+ \partial_t v^\eps(u)+\nabla \cdot f(v^\eps(u))
\\ \nonumber & &
\phantom{d u  = \big[}
-\frac{\eps}{2} \nabla \cdot \big(D(v^\eps(u)) \nabla
v^\eps(u) \big)+G^\eps(u,v^\eps(u))
\big]\,dt 
\\ \nonumber
& & \phantom{d u  =}
+\eps^{\gamma}\, \nabla \cdot
\big[a(u) (\jmath^\eps \ast d W)\big]
\\
& & u(0,x)  =  u^\eps_0(x)
\end{eqnarray}
Suppose 
\begin{itemize}
\item[(i)]{$\lim_\eps \eps^{2(\gamma-1)}\big[\|\jmath^\eps\|_{L_2}^2 +
    \eps\|\nabla\jmath^\eps\|_{L_2}^2\big]=0$.}

\item[(ii)]{There exist adapted processes
    $G_1^\eps,\,G_2^\eps,\,G_3^\eps:\mc \chi \times \mc \chi \to
    L_2([0,T]\times \bb T)$ such that $G^\eps(u,v)(t,x)
    =G_1^\eps(u,v)(t,x) + \nabla \cdot G_2^\eps(u,v)(t,x) + \nabla
    \cdot G_3^\eps(u,v)(t,x)$, and
\begin{equation*}
  |G_3^\eps(u,v^\eps(u))(t,x)| \le G_4^\eps(u)(t,x) |u-v^\eps(u)|
       \qquad \text{for
    $\bb Q^\eps$ a.e.\ $u$}
\end{equation*}
for some adapted process $G_4^\eps: \mc X \to L_2([0,T]\times \bb
T)$.}

\item[(iii)]{Let $G_1,\,G_2$ be as in (ii). Then for each $\delta>0$
\begin{eqnarray*}
  \nonumber
  & & \lim_\eps \bb
  Q^\eps\big(\|v^\eps(u)(0)-u^\eps_0\|_{L_1(\bb T)} + \|G_1^\eps(u,v^\eps(u))
  \|_{L_1([0,T]\times \bb
    T)}
  \\ & & \phantom{\lim_\eps Q^\eps\big(}
  + \eps^{-1}\|G_2^\eps(u,v^\eps(u))\|_{L_2([0,T]\times \bb
    T)}> \delta\big)=0
\end{eqnarray*}
}
 
\item[(iv)]{Let $G_4$ be as in (ii). Then
  \begin{eqnarray*}
  \lim_{\ell \to +\infty} \varlimsup_\eps \bb Q^\eps
 \big(\|G_4^\eps(u,v^\eps(u))
  \|_{L_2([0,T]\times \bb
  T)} +\eps \|\nabla u \|_{L_2([0,T]\times \bb
  T)} >\ell \big)=0
  \end{eqnarray*}
  }
\end{itemize}
Then for each $\delta>0$
\begin{eqnarray}
\label{e:l1stab}
\lim_\eps \bb Q^\eps\big( \|u-v^\eps(u)\|_{L_\infty([0,T];L_1(\bb
  T))}>\delta\big)=0
\end{eqnarray}
\end{lemma}

\begin{proof}
  We denote by $z^\eps(t,x)\equiv z^\eps(u)(t,x) :=
  u(t,x)-v^\eps(u)(t,x) \in [-1,1]$. Let $l \in C^2([-1,1])$. For each
  $\eps,\,t>0$ let us define (in the following we omit the dependence
  of $v^\eps$ and $z^\eps$ on the $u$ variable)
\begin{eqnarray*}
\label{e:extl1}
\nonumber
N^{\eps;l}(t,u) & := & \int\!dx \, [l(z^\eps(t))-l(z^\eps(0))]-
\int_0^t\!ds \Big[ \langle l''(z^\eps) \nabla z^\eps, 
               f(u)-f(v^\eps) \rangle 
\\ \nonumber & &
- \frac{\eps}2 \langle l''(z^\eps)\nabla z^\eps, 
  D(v^\eps)\nabla z^\eps \rangle  
 -\frac{\eps}2 \langle l''(z^\eps)\nabla z^\eps, 
                  [D(u)-D(v^\eps)]\nabla u^\eps \rangle 
\\ \nonumber & &
+\langle l'(z^\eps), G_1^\eps(u,v^\eps)\rangle
-\langle l''(z^\eps)\nabla z^\eps, G_2^\eps(u,v^\eps)\rangle
\\ \nonumber & &
-\langle l''(z^\eps)\nabla z^\eps, G_3^\eps(u,v^\eps)\rangle
+ \frac{\eps^{2\gamma}}{2} \|\nabla \jmath^\eps \|_{L_2(\bb T)}^2
    \langle l''(z^\eps) a(u), a(u)\rangle
\\ & &
 + \frac{\eps^{2\gamma}}{2}  \|\jmath^\eps \|_{L_2(\bb T)}^2
      \langle l''(z^\eps) \nabla u, 
              [a'(u)]^2 \nabla u\rangle
\Big]
\end{eqnarray*}
By It\^{o} formula, $N^{\eps;l}$ is a $\bb Q^\eps$-martingale starting
at $0$, and applying Young inequality for convolutions (analogously to
\eqref{e:youngquad})
\begin{eqnarray}
\label{e:Nlquad}
 \big[N^{\eps,l},N^{\eps,l}\big](t,u) \le \eps^{2\gamma}
\| a(u) l''(z^\eps) \nabla z^\eps\|_{L_2([0,t] \times \bb T)}^2
\end{eqnarray}

We now choose $l$ convex and define
\begin{eqnarray*}
R^{\eps,l}(t)\equiv R^{\eps,l}(u)(t):= \big[\int\!dx\,\langle
l''(z^\eps(t))\nabla
z^\eps(t),\nabla z^\eps(t)\rangle \big]^{1/2}
\end{eqnarray*}
Since $D$ and $f$ are Lipschitz, and $D$ is uniformly positive, by
\eqref{e:extl1} and Cauchy-Schwartz inequality we gather
\begin{eqnarray}
\label{e:extim1}
\nonumber
& &   \int\!dx\,l(z^\eps(t)) -l(z^\eps(0))\le
 -c\, \eps\, [R^{\eps;l}(t)]^2
  \|\sqrt{l''(z^\eps)}z^\eps\|_{L_\infty([0,T]\times \bb T)}
  R^{\eps;l}(t)
\\ \nonumber & &
\qquad \quad
+ C_1\eps \| \nabla u\|_{L_2([0,T]\times \bb T)}\,
 \|\sqrt{l''(z^\eps)}z^\eps\|_{L_\infty([0,T]\times \bb T)}
  R^{\eps;l}(t)
\\ \nonumber & & \qquad \quad
+  \|l'(z^\eps)\|_{L_\infty([0,T]\times \bb T)} \,
\|G_1^\eps(u,v^\eps)\|_{L_1([0,T]\times \bb T)}
\\ \nonumber
 & & \qquad \quad
+\|G_2^\eps(u,v^\eps)\|_{L_2([0,T]\times \bb T)} 
 \|\sqrt{l''(z^\eps)}\|_{L_\infty([0,T]\times \bb T)}
R^{\eps;l}(t)
\\ \nonumber & &
\qquad \quad
+ \|G_4^\eps(u)\|_{L_2([0,T]\times \bb R)}\,
 \|\sqrt{l''(z^\eps)}z^\eps\|_{L_\infty([0,T]\times \bb T)}
  R^{\eps;l}(t)
\\ \nonumber & &
\qquad \quad
+ C_1 \eps^{2\gamma} \|\nabla \jmath^\eps \|_{L_2(\bb T)}^2
    \|l''(z^\eps)\|_{L_\infty([0,T]\times \bb T)}
\\ & &
\qquad \quad
 + C_1 \eps^{2\gamma} \|\jmath^\eps \|_{L_2(\bb T)}^2
 \| l''(z^\eps)\|_{L_\infty([0,T]\times \bb T)}\| \nabla u\|_{L_2([0,T]\times
\bb T)}^2
+N^{\eps;l}(t)
\end{eqnarray}
for some constants $c,\,C_1>0$ independent of $\eps$ and $l$. For
arbitrary $\zeta>0$ to be chosen below, we now consider
$l(Z)=\sqrt{Z^2+\eps^2\zeta^2}$ so that
\begin{eqnarray*}
 |Z|\le l(Z) \le |Z|+\eps \zeta & & \qquad
\max_{Z \in [-1,1]} |l'(Z)| \le 1
\\ 
 \max_{Z \in [-1,1]} |l''(Z)| \le \eps^{-1} \zeta^{-1}
& & \qquad \max_{Z \in
[-1,1]} |l''(Z)\,Z^2|\le \sqrt{2} \eps \zeta
\end{eqnarray*}
Using these bounds in the right hand side of \eqref{e:extim1}, we get
for some $C_2>0$
\begin{eqnarray}
\label{e:extim2}
\nonumber
& &   \int\!dx\,|z^\eps(t)| \le \int\!dx\,|z^\eps(0)| + C_2
\|G_1^\eps\|_{L_1([0,T]\times \bb T)} 
\\ \nonumber & &
\qquad
+C_2 \big[1 + \eps^2 \|\nabla u\|_{L_2([0,T]\times \bb T)}^2
+\|G_4^\eps(u)\|_{L_2([0,T]\times \bb T)}^2\big] \zeta
\\ \nonumber  & &
\qquad
+C_2\zeta^{-1} \big[ \eps^{-2} \|G_2^\eps\|_{L_2([0,T]\times \bb T)}^2
+\eps^{2\gamma-1}  \|\nabla \jmath^\eps
\|_{L_2(\bb
  T)}^2
\\ \nonumber & &
\qquad \phantom{+C_2\zeta^{-1} \big[ }
+\eps^{2(\gamma-1)}  \|\jmath^\eps \|_{L_2(\bb T)}^2 \|\nabla u
\|_{L_2([0,T]\times
\bb
  T)}^2\big]
\\ & &
\qquad
 - \frac{c\,\eps}2 [R^{\eps;l}(t)]^2 + N^{\eps;l}(t)
\end{eqnarray}
where we have also used the straightforward inequality $\alpha R -
\frac{c\eps}2 R^2 \le \frac{\alpha}{2c\eps}$ for a suitable $\alpha \in
\bb R$.

Recalling \eqref{e:Nlquad}, for some $C_3>0$ independent of
$\eps,\,\zeta$
\begin{eqnarray*}
 \big[N^{\eps;l},N^{\eps;l}\big](t,u) \le C_3\, \eps^{2\gamma}
\zeta^{-1} [R^{\eps;l}(t)]^2
\end{eqnarray*}
so that, by maximal inequality for positive supermartingales
\cite[page~58]{RY}, for each $\delta>0$ the term in the last line of
\eqref{e:extim2} satisfies
\begin{eqnarray}
\label{e:martexp2}
\nonumber
& & 
\bb Q^\eps \big (\sup_{s \le t} N^{\eps;l}(s)
 - \frac{c\,\eps}2 [R^{\eps;l}(s)]^2
>\delta \big) \le
\\ \nonumber
& &
\bb Q^\eps \Big(\sup_{s \le t} \exp\big(\frac{2\,c}{C_3}
\eps^{1-2\gamma} \zeta\, N(s) -\frac{2\,c^2}{C_3^2}
\eps^{2(1-2\gamma)} \zeta^2\, [N,N](s)\big) >
\\  
& & \phantom{\phantom{\le}  \bb Q^\eps \big(} 
\exp(\frac{2\,c}{C_3}
\eps^{1-2\gamma} \zeta\, \delta)
\Big)
  \le \exp(-\frac{2\,c}{C_3}
\eps^{-2\gamma+1} \zeta\, \delta)
\end{eqnarray}
Furthermore for $\ell>0$
\begin{eqnarray*}
\nonumber
& & \bb Q^\eps\big(
\sup_t \int\!dx\,|z^\eps(t)| >\delta \big) \le 
\bb Q^\eps\Big(
\sup_t \int\!dx\,|z^\eps(t)| >\delta, 
\\ & & \qquad \qquad 
\|G_4^\eps(u,v^\eps(u))
  \|_{L_2([0,T]\times \bb
  T)}+\eps \|\nabla u \|_{L_2([0,T]\times \bb
  T)}  \le \ell \Big) + o_{\ell,\eps}
\end{eqnarray*}
where $\lim_\ell \varlimsup_\eps o_{\ell,\eps}=0$ by hypotheses (iv).
Therefore, using hypotheses (i) and (iii) and the estimate
\eqref{e:martexp2} in \eqref{e:extim2}, the result easily follows as
we let $\eps \to 0$, then $\zeta \to 0$ and finally $\ell \to
+\infty$.  
\end{proof}

The following result will be used to provide
exponential tightness in stronger topologies in the next sections.
\begin{lemma}
\label{l:exptight1}
There exists a sequence $\{K_\ell\}$ of compact subsets of
$C\big([0,T]; U\big)$ such that
\begin{eqnarray*}
  \lim_\ell \varlimsup_\eps \eps^{2\gamma} \log \bb P^\eps(K_\ell^c)=-\infty
\end{eqnarray*}
\end{lemma}
\begin{proof}
  We refer to the criterion in \cite[Corollary~4.17]{FK} to establish
  the exponential tightness of $\{\bb P^\eps\}$. Let $d \in
  C^1([0,1])$ be any antiderivative of $D$. Integrating twice by parts
  the diffusive term in the weak formulation of \eqref{e:2.2} (see
  \eqref{e:A2} and \eqref{e:A3}), for each $\varphi \in C^\infty(\bb
  T)$ the map $E^{\eps;\varphi}:[0,T]\times C\big([0,T]; U\big) \to
  \bb R$ defined by
\begin{eqnarray*}
  & & E^{\eps;\varphi}(t;u):=\exp\Big[
  \eps^{-2\gamma} \langle u(t),\varphi \rangle -  
  \eps^{-2\gamma} \langle u(0),\varphi \rangle
\\ & & \qquad \qquad
 - \eps^{-2\gamma} \int_0^t \! ds\, \langle f(u) 
           +\frac{\eps}2 d(u),\Delta \varphi \rangle
-\frac 12 \langle \jmath \ast (a(u) \nabla \varphi),
               \jmath \ast (a(u) \nabla \varphi) \rangle
\Big]
\end{eqnarray*}
is a martingale. For a fixed $\varphi \in C^\infty(\bb T)$, the
following bound on the integral term in the definition of
$E^{\eps;\varphi}$ is easily established
\begin{equation*}
\sup_{v \in U} \,\Big| \langle f(v) +\frac{\eps}2
d(v),\Delta \varphi \rangle -\frac 12 \langle \jmath \ast (a(v) \nabla
\varphi), \jmath \ast (a(v) \nabla \varphi) \rangle \Big| <+\infty
\end{equation*}
Furthermore the family of maps $l^\varphi : U \ni v \to \langle
v,\varphi \rangle \in \bb R$ is closed under addition, separates
points in $U$ and satisfies $c\, l^{\varphi}= l^{c\varphi}$ for $c \in
\bb R$. All the hypotheses of the criterion in
\cite[Corollary~4.17]{FK} are therefore satisfied.  
\end{proof}

\begin{proof}[Proof of Proposition~\ref{t:uconv}]
with $\bb Q^\eps \equiv \bb P^\eps :=P\circ (u^\eps)^{-1}$, and
$v^\eps$ as the solution to the (deterministic) Cauchy problem
\begin{eqnarray*}
& & \partial_t v = - \nabla \cdot f(v)+\frac{\eps}{2} \nabla \cdot \big( D(v)
\nabla
v \big)
\\ 
& & v(0,x) = u_0(x)
\end{eqnarray*}
$\bb P^\eps$ and $v^\eps$ fulfill the hypotheses Lemma~\ref{l:stab},
since $G^\eps \equiv 0$ and Lemma~\ref{l:supexpnabla} holds (with
$E^\eps \equiv 0$). As well known \cite[Chap.~6.3]{Daf}, $v^\eps \to
\bar{u}$ in $L_p([0,T]\times \bb T)$. Therefore the statement of the
proposition follows by the same Lemma~\ref{l:stab} and the fact that
$\bb P^\eps$ is (exponentially) tight in $C\big([0,T]; U\big)$, as
proved in Lemma~\ref{l:exptight1}.  
\end{proof}

\subsection{Large deviations with speed $\eps^{-2\gamma}$}
\label{ss:3.2}
In this section we prove Theorem~\ref{t:ld1}.
\begin{lemma}
\label{l:exptight2}
There exists a sequence $\{\mc K_\ell\}$ of compact
subsets of $\mc M$ such that
\begin{equation}
\label{e:exptight2}
  \lim_\ell \varlimsup_\eps \eps^{2\gamma} 
      \log \mb P^\eps(\mc K_\ell^c)=-\infty
\end{equation}
\end{lemma}
\begin{proof}
  Let the sequence $\{K_\ell\}$ of compact subsets of $C\big([0,T];
  U\big)$ be as in Lemma~\ref{l:exptight1}. For $\ell>0$ consider the
  set
\begin{eqnarray*}
  \tilde{\mc K}_\ell:= \{\mu \in \mc M\,:\:\mu_{t,x}=\delta_{u(t,x)}\,
   \text{for some $u \in K_\ell$}  \}
\end{eqnarray*}
Then $\mb P^\eps(\tilde{\mc K}_\ell)=\bb P^\eps(K_\ell)$ and by
Lemma~\ref{l:exptight1}
\begin{equation*}
\lim_\ell \varlimsup_\eps \eps^{2\gamma} 
\log \mb P^\eps(\tilde{\mc K}_\ell^c)=-\infty
\end{equation*}
On the other hand $\tilde{\mc K}_\ell$ is precompact in $(\mc M,d_{\mc
  M})$ for any $\ell$, and thus the Lemma is proved by taking $\mc
K_\ell$ to be the closure of $\tilde{\mc K}_\ell$.  
\end{proof}

\begin{proof}[Proof of Theorem~\ref{t:ld1}: upper bound]
Let $d \in C^2([0,1])$ be any antiderivative of $D$. For $\eps>0$ and
$\varphi \in C^\infty([0,T]\times \bb T)$, define the map $ \mc
N^{\eps;\varphi} : [0,T]\times \mc M \to \bb R$ by
\begin{eqnarray*}
\nonumber
  \mc N^{\eps;\varphi}(t,\mu) & := &
              \langle \mu_{T,\cdot}(\imath),\varphi(T)\rangle 
             -\langle u_0,\varphi(0) \rangle
\\
\phantom{ \mc N^\varphi  }& \phantom{:} & 
  - \int_0^t\!ds\,\big[ \langle \mu(\imath), \partial_t \varphi \rangle 
  - \langle \mu(f), \nabla \varphi \rangle 
  +\frac{\eps}2  \langle \mu(d),\Delta \varphi \rangle \big]
\end{eqnarray*}
$\mb P^\eps$ is concentrated on the set 
\begin{equation*}
  \{\mu \in \mc M\,:\:
  \mu=\delta_u\,\text{for some $u \in C\big([0,T]; U\big)$}\}
\end{equation*}
so that $\mc N^{\eps;\varphi}$ is a $\mb P^\eps$-martingale. Indeed an
integration by parts shows that $\mc N^{\eps;\varphi}(t,\delta_u)$ is
the martingale term appearing in the very definition of
\emph{martingale solution to \eqref{e:2.2}}, see the appendix.
Reasoning as in \eqref{e:youngquad}, we have 
\begin{equation*}
\big[\mc
N^{\eps;\varphi},\mc N^{\eps;\varphi}\big](t,\mu) \le \eps^{2\gamma}
\int_0^t\!ds\, \langle \mu(a^2)\nabla \varphi, \nabla \varphi\rangle
\end{equation*}
Therefore, the map $\mc Q^{\eps;\varphi} : [0,T]\times \mc M
\to \bb R$ defined by
\begin{eqnarray*}
  \mc Q^{\eps;\varphi}(t,\mu) := \exp \Big\{\mc
N^{\eps;\varphi}(t,\mu)
 -
  \frac{\eps^{2\gamma}}{2} \int_0^t\!ds\, \langle \mu(a^2)\nabla \varphi,
  \nabla \varphi\rangle \Big\}
\end{eqnarray*}
is a continuous $\mb P^\eps$-supermartingale, with $\mc
Q^{\eps;\varphi}(0,\mu)=1$ and $\mc Q^{\eps;\varphi}(T,\mu)>0$, $\mb
P^\eps$ a.s.. For an arbitrary Borel set $\mc A \subset \mc M$ we then
have
\begin{eqnarray*}
  \nonumber
  \mb P^\eps(\mc A) & = & \bb E^{\mb P^\eps}\big(\id_{\mc A}(\cdot)
  \mc Q^{\eps;\varphi}(T,\cdot)
  [\mc Q^{\eps;\varphi}(T,\cdot)]^{-1} \big) 
  \\ 
  &\le & \sup_{\mu \in \mc A} [\mc Q^{\eps;\varphi}(T,\mu)]^{-1}  
       \bb E^{\mb P^\eps}\big(
  \id_{\mc A}(\cdot) \mc Q^{\eps;\varphi}(T,\cdot) \big) 
  \le  \sup_{\mu \in \mc A} [\mc Q^{\eps;\varphi}(T,\mu)]^{-1}
\end{eqnarray*}
Since this inequality holds for each $\varphi$, we can evaluate it
replacing $\varphi$ with $\eps^{-2\gamma}\varphi$, thus obtaining
\begin{eqnarray*}
\nonumber
\eps^{2\gamma} \log \mb P^\eps(\mc A) & \le & - \inf_{\mu \in \mc
A} \Big\{ \langle \mu_{T,\cdot}(\imath),\varphi(T)\rangle 
             -\langle u_0,\varphi(0) \rangle
  - \langle \langle \mu(\imath), \partial_t \varphi \rangle \rangle
\\ \nonumber
& & \quad 
  - \langle \langle \mu(f), \nabla \varphi \rangle \rangle
  -\frac{\eps}2  \langle \langle \mu(d), \Delta \varphi \rangle \rangle
-\frac 12 \langle \langle \mu(a^2)\nabla \varphi,
\nabla \varphi\rangle\rangle
\Big\}
\\ \nonumber
 &\le &
 -  \inf_{\mu \in \mc A} \Big\{ 
 \langle \mu_{T,\cdot}(\imath),\varphi(T)\rangle 
             -\langle u_0,\varphi(0) \rangle
  - \langle \langle \mu(\imath), \partial_t \varphi \rangle \rangle
\\
& & \quad
  - \langle \langle \mu(f), \nabla \varphi \rangle \rangle
 -\frac 12 \langle \langle \mu(a^2)\nabla \varphi,
\nabla \varphi\rangle\rangle
\Big\} + \eps\, C_{d,\varphi}
\end{eqnarray*}
for some constant $C_{d,\varphi}$ depending only on $d$ and
$\varphi$. Taking the limsup for $\eps \to 0$, the last term
vanishes. Optimizing on $\varphi$:
\begin{eqnarray*}
\varlimsup_\eps \eps^{2\gamma} \log \mb P^\eps(\mc A)
& \le &
-\sup_{\varphi \in C^\infty([0,T]\times \bb T)} \inf_{\mu \in \mc A}
\Big\{
\langle \mu_{T,\cdot}(\imath),\varphi(T)\rangle 
             -\langle u_0,\varphi(0) \rangle
\\ & & \quad
  - \langle \langle \mu(\imath), \partial_t \varphi \rangle \rangle
  - \langle \langle \mu(f), \nabla \varphi \rangle \rangle
 -\frac 12 \langle \langle \mu(a^2)\nabla \varphi,
\nabla \varphi\rangle\rangle
\Big\}
\end{eqnarray*}
By a standard application \cite[Appendix~2, Lemma~3.2]{KL} of the
minimax lemma, we gather that upper bound with rate $\mc I$, see
\eqref{e:2.5}, holds on each compact subset $\mc K \subset \mc M$. By
Lemma~\ref{l:exptight2}, it holds on each closed subset of $\mc M$.
\end{proof}

We recall a well known method to prove large deviations lower bounds,
see e.g.\ \cite[Chap.~4]{J}. For $\bb P$, $\bb Q$ two Borel
probability measures on a Polish space, we denote by $\mathrm{Ent}(\bb
Q|\bb P)$ the relative entropy of $\bb Q$ with respect to $\bb P$.
\begin{lemma}
\label{l:ldglim}
Let $\mc X$ be a Polish space, $I:\mc X\to [0,+\infty]$ a positive
functional, $\{\alpha_\eps\}$ a sequence of positive reals such that
$\lim_\eps \alpha_\eps=0$, and let $\{\bb P^\eps\} \subset \mc P(\mc
X)$. Suppose that for each $x \in \mc X$ there is a sequence $\{\bb
Q^{\eps,x}\} \subset \mc P(\mc X)$ such that $\bb Q^{\eps,x} \to
\delta_x$ weakly in $\mc P(\mc X)$, and $\varlimsup_\eps \alpha_\eps
\mathrm{Ent}_\eps(\bb Q^{\eps,x}|\bb P^\eps) \le I(x)$. Then $\{\bb
P^\eps\}$ satisfies a large deviations lower bound with speed
$\alpha_\eps^{-1}$ and rate $I$.
\end{lemma}

\begin{proof}[Proof of Theorem~\ref{t:ld1}: lower bound]
We will prove the lower bound following the strategy suggested by
Lemma~\ref{l:ldglim}. More precisely, consider the set
\begin{equation*}
\mc M_0:=\Big\{\mu \in \mc M\,:\: \exists \zeta>0\,:\:\mu =\delta_u \text{ for
some $u \in  C^2\big([0,T]\times \bb T;[\zeta,1-\zeta]\big)$} 
\Big\}
\end{equation*}
Here we prove that for each $\mu \in \mc M_0$ there exists a sequence
of probability measures $\{\mb Q^\eps\} \subset \mc P(\mc M)$ such
that $\mb Q^\eps \to \delta_\mu$ and $\varlimsup
\eps^{2\gamma}\mathrm{Ent}(\mb Q^\eps|\mb P^\eps) \le \mc I(\mu)$. By
Lemma~\ref{l:ldglim} this will yield a large deviations lower bound
with rate $\tilde{\mc I}:\mc M \to [0,+\infty]$ defined as
\begin{eqnarray*}
\tilde{\mc I}(\mu):=
\begin{cases}
\displaystyle{ \mc I(\mu) } & \text{if $\mu \in \mc M_0$}
\\
+\infty & \text{otherwise}
\end{cases}
\end{eqnarray*}
By a standard diagonal argument, the lower bound then also holds with
the lower semicontinuous envelope of $\tilde{\mc I}$ as rate
functional. In \cite[Theorem~4.1]{BBMN} it is shown, in a slightly
different setting, that the lower semicontinuous envelope of
$\tilde{\mc I}$ is indeed $\mc I$. By the assumption $\zeta \le u_0
\le 1-\zeta$ (which is equivalent to the requirement that $a^2(u_0)$
is uniformly positive), it is not difficult to adapt the arguments in
the proof of Theorem~4.1 in \cite[Theorem~4.1]{BBMN}, to obtain the
analogous result in this case. We are thus left with the proof of the
lower bound on $\mc M_0$.

Let $\mu \in \mc M_0$ be such that $\mc I(\mu)<\infty$. Then
$\mu=\delta_v$ for some smooth $v \in C\big([0,T]; U\big)$ with
$v(0,x)=u_0(x)$ and $a(v)^2 \ge r$ for some $r>0$. By the definition
of $\mc I$ and the smoothness of $v$
\begin{eqnarray*}
  \mc I(\mu) & = &
\sup_{\varphi \in C^\infty\big([0,T] \times \bb T\big) } \big\{ 
  - \langle\langle \partial_t v +\nabla \cdot f(v), \varphi \rangle\rangle 
   - \frac 12\, \langle\langle a(v)^2 \nabla \varphi,
  \nabla \varphi \rangle\rangle \big\}
\\
& \ge &
\sup_{\varphi \in C^\infty\big([0,T] \times \bb T\big) } \big\{ 
  - \langle\langle \partial_t v +\nabla \cdot f(v), \varphi \rangle\rangle 
   - \frac r2 \langle\langle \nabla \varphi, 
  \nabla \varphi \rangle\rangle \big\}
\end{eqnarray*}
Since the supremums in this formula are finite, Riesz representation
lemma implies the existence of a $\Psi^v \in L_2\big([0,T]; H^1(\bb
T)\big)$ such that
\begin{eqnarray}
\label{e:equ}
  \partial_t v + \nabla \cdot f(v)= -\nabla \cdot [a(v)^2 \nabla \Psi^v]
\end{eqnarray}
holds weakly and
\begin{eqnarray}
\label{e:Imu}
  \mc I(\mu)=\frac 12 \langle \langle a(v)^2 \nabla \Psi^v, 
       \nabla \Psi^v\rangle\rangle
\end{eqnarray}
We next define the $P$-martingale $M^{\eps;v}$ on $\Omega$ as
\begin{eqnarray*}
  M^{\eps;v}(t):=-\eps^{-\gamma}\int_0^t \big\langle 
  \jmath^\eps \ast [a(v) \nabla \Psi^v], dW \big] \big\rangle
\end{eqnarray*}
so that, by Young inequality for convolutions and \eqref{e:Imu}, we
have $P$ a.s.
\begin{equation}
\label{e:Mvarineq}
\big[M^{\eps;v},M^{\eps;v}\big](T)  \le  \eps^{-2 \gamma}
\| a(v) \nabla \Psi^v
\|_{L_2([0,T] \times \bb T)}^2
= 2\eps^{-2\gamma} \mc I(\mu)
\end{equation}
Since the quadratic variation of $M^{\eps;v}$ is bounded, its
stochastic exponential 
\begin{equation*}
E^{\eps;v}(t,\omega):=\exp\big(M^{\eps;v}(t,\omega)- \frac
12 [M^{\eps;v},M^{\eps;v}](t,\omega) \big)
\end{equation*}
is a uniformly integrable $P$-martingale. For $\eps>0$ we define the
probability measure $Q^{\eps;v}$ on $\Omega$ by
\begin{equation*}
Q^{\eps;v}(d \omega) := E^{\eps;v}(T,\omega) P(d \omega)
\end{equation*}
Recalling that $u^\eps$ was the process solving \eqref{e:2.2}, we next
define $\mb Q^{\eps;v} :=Q^{\eps;v} \circ (\delta_{u^\eps})^{-1} \in
\mc P(\mc M)$. Then
\begin{eqnarray}
\label{e:entropy1}
\nonumber
\eps^{2\gamma}\mathrm{Ent}(\mb Q^{\eps;v}|\mb P^{\eps;v})
& \le & \eps^{2\gamma}\mathrm{Ent}(Q^{\eps;v}| P)
=
 \eps^{2\gamma} \int Q^{\eps;v}(d\omega) \log E^{\eps;v}(T,\omega)
\\ \nonumber
& = &
\eps^{2\gamma} \int Q^{\eps;v}(d\omega) \big(M^{\eps;v}(T,\omega) -
[M^{\eps;v},M^{\eps;v}](T,\omega) \big) 
\\
& &
+ \frac{\eps^{2 \gamma}}2 
\int Q^{\eps;v}(d\omega) [M^{\eps;v},M^{\eps;v}](T,\omega)
 \le I(\mu)
\end{eqnarray}
where in the last line we used Girsanov theorem, stating that
$M^{\eps;v}-[M^{\eps;v},M^{\eps;v}]$ is a $Q^{\eps,v}$-martingale and
it has therefore vanishing expectation, and \eqref{e:Mvarineq}.

By \eqref{e:entropy1}, Lemma~\ref{l:exptight2} and entropy inequality,
the sequence $\{\mb Q^{\eps;v}\}$ is tight in $\mc P(\mc M)$, and in
view of \eqref{e:entropy1} it remains to show that any limit point of
$\{\mb Q^{\eps;v}\}$ is concentrated on $\{\delta_v\}$. Let $\bb
Q^{\eps;v}:= Q^{\eps;v} \circ (u^\eps)^{-1} \in \mc P\big(C\big([0,T];
U\big)\big)$; we will show
\begin{eqnarray}
 \label{e:ld1lbconv}
 \lim_\eps  \bb E^{\bb Q^\eps}
 \big(\sup_t \|u(t)-v(t)\|_{L_1(\bb T)}\big)=0
\end{eqnarray}
which is easily seen to imply the required convergence of $\{\mb
Q^\eps\}$.  Since $\bb Q^{\eps;v}$ is absolutely continuous with
respect to $\bb P^\eps$, it is concentrated on $C\big([0,T]; U\big)
\cap L_2\big([0,T];H^1(\bb T)\big)$ and by Girsanov theorem it is a
solution to the martingale problem associated with the stochastic
partial differential equation in the unknown $u$
\begin{eqnarray}
\label{e:stoc2}
\nonumber
& & d u =  \Big[- \nabla \cdot f(u)
  +\frac{\eps}{2} \nabla \cdot \big [D(u) \nabla
  u
  - a(u) \big((\jmath^\eps \ast \jmath^\eps) \ast 
(a(v)  \nabla \Psi^v)\big)\big]
  \Big]\,dt 
  \\  \nonumber
& &  \phantom{d u =}
  +\eps^{\gamma}\, \nabla \cdot
  \big[a(u) (\jmath^\eps \ast d W)\big]
  \\
& &   u(0,x) =  u_0^\eps(x)
\end{eqnarray}
where we used the same notation of \eqref{e:2.2}. Note that $\Psi^v$
is twice continuously differentiable, since $a(v)^2$ is strictly
positive and \eqref{e:equ} can be regarded as an elliptical equation
for $\Psi^v$ with smooth data. Therefore by Lemma~\ref{l:supexpnabla}
applied with $E^\eps= \jmath^\eps \ast \jmath^\eps \ast [a(v) \nabla
\Psi^v]$ we have that $ \bb E^{\bb Q^{\eps;v}} \big( \eps \|
\nabla u\|_{L_2([0,T]\times \bb T)}^2 \big)$ is bounded uniformly in
$\eps$. By \eqref{e:equ} and \eqref{e:stoc2}, we can then apply
Lemma~\ref{l:stab} with: $v^\eps(u)(t,x)=v(t,x)$,
$G_1^\eps(u,v)(t,x)=\frac{\eps}2 \nabla \cdot \big[D(v) \nabla v
\big]$, $G_2^\eps(u,v)=0$ and $G_3^\eps(u,v)(t,x)=[a(v)-a(u)]
\big[a(v) \Psi^v - \jmath^\eps\ast \jmath^\eps \ast [a(v)\Psi^v]
\big]$. Since $v$ and $\Psi^v$ are smooth, the hypotheses of
Lemma~\ref{l:stab} hold and we thus obtain \eqref{e:ld1lbconv}.  
\end{proof}

\begin{proof}[Proof of Corollary~\ref{c:ld1}]
The corollary is an immediate consequence of the contraction principle
\cite[Theorem~4.4.1]{DZ2} applied to the continuous map $\mc M \ni \mu
\mapsto \mu(\imath) \in C\big([0,T]; U\big)$. If $\mu \in \mc M$ is
such that $\mc I(\mu)<\infty$, then there exists $\Phi \in
L_2\big([0,T]\times \bb T\big)$ such that $\partial_t \mu(\imath)=
-\nabla \Phi$ holds weakly, and thus we have for $u \in C\big([0,T];
U\big)$ and any $\Phi$ in the above class
  \begin{eqnarray*}
 & &  \inf_{\mu \in \mc M,\,\mu(\imath)=u} I(\mu) =
    \inf_{\mu \in \mc M,\,\mu(\imath)=u} 
\\ & & \qquad \qquad \qquad\qquad
\sup_{\varphi \in C^\infty([0,T] \times \bb T) } \Big\{
  \langle\langle \Phi-\mu (f), \nabla \varphi \rangle\rangle 
  - \frac 12\, \langle\langle \mu (a^2) \nabla \varphi, 
     \nabla \varphi \rangle\rangle 
  \Big\}
\\ & & \qquad
= \int \!dt\, \inf_{c\in \bb R} \int \!dx\,
 \inf_{\mu \in \mc M,\,\mu(\imath)=u}
      \frac{\big[\mu_{t,x}(f)-\Phi(t,x)-c \big]^2}{\mu_{t,x}(a^2)}
  \end{eqnarray*}
  Since the function $\Phi$ satisfying $\partial_t \mu(\imath)=
  -\nabla \Phi$ are defined up to a measurable additive function of
  $t$, the optimization over $c$ can be replaced by an optimization
  over $\Phi$, namely
 \begin{eqnarray*}
 & &  \inf_{\mu \in \mc M,\,\mu(\imath)=u} I(\mu) =
\inf_{\Phi \in L_2([0,T]\times \bb T), \nabla \Phi=-\partial_t u} 
\\ & & \qquad \qquad \qquad \qquad \quad
 \inf_{\mu \in \mc M,\,\mu(\imath)=u}
\int \!dt\, dx\,
      \frac{\big[\mu_{t,x}(f)-\Phi(t,x)\big]^2}{\mu_{t,x}(a^2)}
  \end{eqnarray*}
which coincides with $I(u)$. 
 \end{proof}

\subsection{Large deviations with speed $\eps^{-2\gamma+1}$}
\label{ss:3.3}
The next statement follows easily from entropy inequality (see also
the introduction of \cite{M} for further details).
\begin{lemma}
\label{p:ldglim1}
Let $\mc X$ be a Polish space and $\{\bb P^\eps\} \subset \mc P(\mc
X)$. The following are equivalent:
\begin{itemize}
\item[$(i)$]{ $\{\bb P^\eps\}$ is exponentially tight with speed
$\eps^{-2\gamma+1}$.}

\item[$(ii)$]{If a sequence $\{\bb Q^\eps\} \subset \mc P(\mc X)$ is
such that $\varlimsup_\eps \eps^{2\gamma-1} \mathrm{Ent}(\bb Q^\eps|\bb P^\eps)
<+\infty $, then $\{\bb Q^\eps\}$ is tight.}
\end{itemize}
\end{lemma}

Let $\bb Q \in \mc P(\mc X)$. For $\Psi:[0,T]\times \mc X \to
C^\infty([0,T]\times \bb T)$ a predictable process, let 
\begin{equation*}
\|\Psi\|_{\mc
  D^\eps(\bb Q)}^2:=\int \! \bb Q(du) \big\| \jmath^\eps \ast
[a(u) \nabla \Psi(u)]\big\|_{L_2([0,T] \times \bb T)}^2  \in [0,+\infty]
\end{equation*}
We let $\mc D^\eps(\bb Q)$ be the Hilbert
space obtained by identifying and completing the set of predictable
processes $\Psi:[0,T]\times \mc X \to C^\infty([0,T]\times \bb T)$
such that $\|\cdot\|_{\mc D^\eps(\bb Q)}<+\infty$ with respect to this
seminorm.
\begin{lemma}
\label{l:comp1}
Let $\eps>0$ and $\bb Q \in \mc P(\mc X)$ be such that
$\mathrm{Ent}(\bb Q|\bb P^\eps)<+\infty$. Then there exists $\Psi \in
\mc D^\eps(\bb Q)$ such that $\bb Q$ is a martingale solution to the
Cauchy problem in the unknown $u$
\begin{eqnarray}
\label{e:stocpert}
\nonumber
& & d u =  \big(- \nabla \cdot f(u)
+\frac{\eps}{2} \nabla \cdot \big(D(u) \nabla
u \big)
-\eps^\gamma\nabla \cdot \big[a(u) \jmath^\eps \ast \jmath^\eps \ast [a(u)
\nabla \Psi(u)]\big] \big)dt 
\\ \nonumber
 & & \phantom{ d u =}
+\eps^{\gamma}\, \nabla \cdot
\big[a(u) (\jmath^\eps \ast d W)\big]
\\ 
& & u(0,x) =  u_0^\eps(x)
\end{eqnarray}
and $\mathrm{Ent}(\bb Q|\bb P^\eps) \ge \frac 12 \|\Psi\|_{\mc D^\eps(\bb
Q)}^2$.
\end{lemma}
\begin{proof}
  Since $\bb Q$ is absolutely continuous with respect to $\bb P^\eps$,
  there exists a continuous local $\bb P^\eps$-martingale $N$ on $\mc
  X$ such that 
\begin{equation*}
\bb Q(du)=\exp \big(N(T,u)-\frac 12 \big[N,N
  \big](T,u) \big) \bb P^\eps(du)
\end{equation*}
and $\mathrm{Ent}\big(\bb Q|\bb P^\eps \big)= \bb E^{\bb Q}
\big(N(T)-\frac 12 \big[N,N\big](T)\big) =\frac 12 \bb E^{\bb Q} \big(
\big[N,N\big](T)\big)$ by Girsanov theorem.

  It is easy to see that, as $\varphi$ runs in $C^\infty([0,T]\times
  \bb T)$, the family of maps (defined $\bb P$ a.s.)
\begin{eqnarray*}
& & 
\mc [0,T]\times \chi \ni (t,u) \mapsto \langle M(t,u), \varphi \rangle  := 
\langle u(t),\varphi(t)\rangle - \langle u(0), \varphi(0) \rangle
\\ & & \qquad \qquad
     - \int_0^t\!ds\, \big\langle u,\partial_t \varphi \rangle 
        - \langle f(u) -\frac 12 D(u) \nabla u,\nabla \varphi \big\rangle \in
\bb R
\end{eqnarray*}
generates the standard filtration of $\mc X$. Therefore the martingale
$N$ is adapted to $\{\langle M,\varphi\rangle\}$, and reasoning as in
\cite[Lemma~4.2]{RY}, there exists a predictable process $\Psi$ on
$\mc X$ and a martingale $\tilde{N}$ such that
\begin{eqnarray*}
 N(t)= \int_0^t \langle \Psi, dM\rangle + \tilde{N}(t)
\end{eqnarray*}
and
\begin{eqnarray}
\label{e:quad0}
 \big[\tilde{N},\langle M, \varphi \rangle \big](T,u)=0 \qquad \text{for all
$\varphi
\in C^\infty(\bb T)$, for $\bb P$ a.e. $u$.}
\end{eqnarray}
In particular 
\begin{eqnarray*}
 \bb E^{\bb Q} \big( \big[N,N](T)\big) & = &  \bb E^{\bb Q}
\big(\big[\int_0^\cdot \langle \Psi, dM\rangle ,\int_0^\cdot \langle \Psi,
dM\rangle\big](T) \big) +  \bb E^{\bb Q} \big( \big[\tilde{N},\tilde{N}](T)\big)
\\
& \ge &
 \bb E^{\bb Q} \big(\big\| \jmath^\eps \ast [a(u) \nabla \Psi(u)]
\big\|_{L_2([0,T] \times \bb T)}^2   \big)
\end{eqnarray*}
Therefore $\mathrm{Ent}(\bb Q|\bb P^\eps) \ge \frac 12 \|\Psi\|_{\mc
  D^\eps(\bb Q)}$ and \eqref{e:stocpert} follows by Girsanov theorem
and \eqref{e:quad0}.  It is immediate to see that both the bound on
the relative entropy $\mathrm{Ent}(\bb Q|\bb P^\eps)$ and the Girsanov
term in \eqref{e:stocpert} are compatible with the identification
induced by the seminorm $\|\cdot\|_{\mc D^\eps(\bb Q)}$, and thus one
can identify $\Psi$ with an element in $\mc D^\eps(\bb Q)$.  
\end{proof}

\begin{lemma}
\label{l:exptight3}
Under the same hypotheses of Theorem~\ref{t:ld2} item (i), there
exists a sequence $\{K_\ell \}$ of compact subsets of $\mc
X$ such that
\begin{equation*}
\lim_\ell \varlimsup_\eps \eps^{2\gamma-1} \log \bb P^\eps(K_\ell)=-\infty
\end{equation*}
\end{lemma}
\begin{proof}
  In view of Lemma~\ref{p:ldglim1}, we will prove that if ${\bb
    Q^\eps} \subset \mc P(\mc X)$ is a sequence with $\eps^{2\gamma
    -1} \mathrm{Ent}(\bb Q^\eps|\bb P^\eps) \le C$ for some $C \ge 0$
  independent of $\eps$, then $\bb Q^\eps$ is tight. By
  Lemma~\ref{l:comp1}, there exists a sequence $\Psi^\eps \in \mc
  D^\eps(\bb Q^\eps)$ such that
\begin{eqnarray}
\label{e:PsiD}
\frac{\eps^{-1}}2 \|\Psi^\eps\|_{\mc D^\eps(\bb Q^\eps)}^2
\le \eps^{2\gamma-1}  \mathrm{Ent}(\bb Q^\eps|\bb P^\eps) \le    C
\end{eqnarray}
and $\bb Q^\eps$ is a martingale solution to the Cauchy problem in the
unknown $u$
\begin{eqnarray}
\label{e:stocpert2}
\nonumber
& & d u  =  \Big(- \nabla \cdot f(u)
+\frac{\eps}{2} \nabla \cdot \big(D(u) \nabla
u \big)
\\ \nonumber
 & & \phantom{ d u  = \Big(}
- \nabla \cdot \big[a(u) \jmath^\eps \ast \jmath^\eps \ast
[a(u) \nabla
\Psi^\eps(u)]\big] \Big)\,dt 
+\eps^{\gamma}\, \nabla \cdot
\big[a(u) (\jmath^\eps \ast d W)\big]
\\ 
& & u(0,x)  =  u_0^\eps(x)
\end{eqnarray}
For $\eps>0$, we next define ($\bb P^\eps$ a.s.) the predictable map
$v^\eps:\mc X \to \mc X$ as the solution to the parabolic Cauchy problem
\begin{eqnarray}
\label{e:stocpert3}
\nonumber
& & \partial_t v   =  - \nabla \cdot f(v)
+\frac{\eps}{2} \nabla \cdot \big(D(v) \nabla
v \big)
- \nabla \cdot \big[a(v) \jmath^\eps \ast \jmath^\eps \ast
[a(u) \nabla
\Psi^\eps(u)]\big]
\\ 
& & v(0,x) = u_0(x)
\end{eqnarray}
It is easily seen that, for $\bb P^\eps$ a.e.\ $u$,
\eqref{e:stocpert3} admits a unique solution $v^\eps(u) \in \mc X \cap
L_2\big([0,T];H^1(\bb T)\big)$, and that the definition of $v^\eps$ is
compatible with the equivalence relation for $\Psi^\eps$ in the
definition of $\mc D^\eps(\bb Q^\eps)$. By \eqref{e:stocpert3} and
Young inequality for convolutions we also have
\begin{eqnarray}
\label{e:boundiv5}
\nonumber
I_\eps(v^\eps(u)) & = & \frac 12 \big\|\jmath^\eps \ast\jmath^\eps \ast
[a(u) \nabla \Psi^\eps(u)] \big\|_{L_2([0,T] \times \bb T)}^2
\\
& \le & 
\frac 12 \big\|\jmath^\eps \ast
[a(u) \nabla \Psi^\eps(u)] \big\|_{L_2([0,T] \times \bb T)}^2
\end{eqnarray}
where $I_\eps: \mc X \cap L_2\big([0,T];H^1(\bb T)\big) \to [0,+\infty]$ is
defined as
\begin{eqnarray*}
I_\eps (v) & := &
 \sup_{\varphi\in C^\infty([0,T]\times \bb R)}
\Big[
\langle v(T),\varphi(T)\rangle - \langle u_0, \varphi(0) \rangle
\\ & & - \langle \langle v,\partial_t \varphi \rangle \rangle
       + \langle \langle f(v) -\frac 12 D(v) \nabla v,\nabla \varphi 
        \rangle \rangle
   - \frac 12 \langle\langle a(v)^2 \nabla\varphi, 
                     \nabla \varphi \rangle\rangle \Big]
\end{eqnarray*}
Therefore taking the $\bb E^{\bb Q^\eps}$ expectation in
\eqref{e:boundiv5}, multiplying by $\eps^{-1}$ and using
\eqref{e:PsiD}
\begin{eqnarray}
 \label{e:entbound1}
 E^{\bb Q^\eps} \big(
 \eps^{-1} I_\eps(v^\eps(u))\big) \le
\frac{\eps^{-1}}2 \|\Psi^\eps\|_{\mc D^\eps(\bb Q^\eps)}^2
\le    C
\end{eqnarray}
Minor adaptations of the proof of \cite[Theorem~2.5]{BBMN} imply
that for each $\ell>0$ there exist $\eps_0(\ell)>0$ and a compact
$K_\ell \subset \mc X$ such that
\begin{eqnarray}
\label{e:Hequic}
 \cup_{\eps \le \eps_0(\ell)} \big\{v \in \mc X \cap L_2\big([0,T];H^1(\bb
T)\big)\,:\: \eps^{-1} I_\eps(v) \le \ell \big\} \subset K_\ell
\end{eqnarray}
\eqref{e:entbound1} and \eqref{e:Hequic}  imply that the the sequence
$\{\bb Q^\eps \circ (v^\eps)^{-1} \} \subset \mc P(\mc X)$ is tight in
$\mc X$, since by  Chebyshev inequality 
\begin{equation*}
\big(  \bb Q^\eps \circ (v^\eps)^{-1}\big)(K_\ell^c ) \le C \ell ^{-1}
\end{equation*}
By Lemma~\ref{l:supexpnabla} (applied to $\bb P^\eps$ with $E^\eps
\equiv 0$) and entropy inequality, we have
\begin{eqnarray*}
\lim_{\ell \to +\infty} \varlimsup_\eps \bb Q^\eps \big(\eps \|
\nabla u \| _{L_2([0,T] \times \bb T)}^2 \ge \ell \big) = 0
\end{eqnarray*}
Therefore, in view of \eqref{e:stocpert2} and \eqref{e:stocpert3} we can
apply Lemma~\ref{l:stab} to $\bb Q^\eps$ with $G_1(u,v)=0$, $G_2(u,v)=0$,
$G_3(u,v)=[a(v)-a(u)]\big[\jmath^\eps \ast \jmath^\eps \ast [a(u) \nabla
\Psi^\eps(u
)\big]$. Indeed, since \eqref{e:PsiD} holds, the hypotheses of
Lemma~\ref{l:stab} are easily satisfied. We then gather for each $\delta>0$
\begin{eqnarray*}
 \lim_\eps \bb Q^\eps\big(\sup_t \|u-v^\eps(u)\|_{L_1(\bb T)} \ge \delta \big)
=0
\end{eqnarray*}
which implies, together with the tightness of $\{\bb Q^{\eps} \circ
(v^\eps)^{-1}\}$ proved above, the tightness of $\{\bb Q^\eps\}$.
\end{proof}

\begin{proof}[Proof of Theorem~\ref{t:ld2}: upper bound]
  Let $\mc W \subset \mc X$ be the set of weak solutions to
  \eqref{e:2.4}. Let $K \subset \mc X$ be compact, and set $\mc K:= \{
  \mu \in \mc M\,:\:\mu=\delta_u,\,\text{for some $u \in K$}\}$. $\mc
  K$ is compact in $\mc M$, since $\mc X$ is equipped with the
  topology induced by the map $\mc X \ni u \mapsto \delta_u \in \mc
  M$. If $K \cap \mc W = \emptyset$, then $\inf_{\mu \in \mc K} \mc
  I(\mu) >0$ as $\mc I$ vanishes only on measure-valued solutions to
  \eqref{e:2.4}. In particular by Theorem~\ref{t:ld1} item (i)
\begin{eqnarray*}
\varlimsup_\eps \eps^{2\gamma-1} \log \bb P^\eps(K)=\varlimsup_\eps
\eps^{2\gamma-1} \log \mb P^\eps(\mc K)=-\infty
\end{eqnarray*}
Then, since $\mc W$ is closed in $\mc X$ and Lemma~\ref{l:exptight3}
holds, we need to prove the large deviations upper bound for $\{\bb
P^\eps\}$ only for compact sets $K \subset \mc W \subset \mc X$.

Let $(\vartheta,Q)$ be an entropy sampler--entropy sampler flux
pair. Recall the definition of the martingale $N^{\eps;\vartheta}$ in
Lemma~\ref{l:ito}, and consider its stochastic exponential
\begin{eqnarray*}
& & E^{\eps;\vartheta}(t,u)  := \exp \big(N^{\eps,\vartheta}(t,u) -
\frac 12 \big[N^{\eps,\vartheta},N^{\eps,\vartheta} \big](t,u)\big)
\\ & & \quad 
=  \exp \Big\{
  \int\!dx\, \vartheta(u (t),t,x)
  -\int\!dx\, \vartheta(u_0,0,x)
\\ & & \qquad
  -\int_0^t\!ds \int \! dx\,
  \big[\big(\partial_s \vartheta)\big(u (s,x),s,x \big) 
  + \big(\partial_x Q\big)\big(u (s,x),s,x \big)\big]
  \\  & &\qquad
  + \int_0^t\!ds\,\Big[ \frac{\eps}2 \langle  \vartheta''(u ) \nabla u ,
  D(u ) \nabla u  \rangle 
  +\frac{\eps}2 \langle \partial_x \vartheta'(u ),
  D(u ) \nabla u \rangle
  \\ & & \qquad \phantom{ + \int_0^t\!ds\,\Big[ }
  - \frac{\eps^{2\gamma}}{2} \|\nabla \jmath^\eps \|_{L_2(\bb T)}^2
    \langle \vartheta''(u ) a(u ), a(u )\rangle 
  \\ & & \qquad \phantom{ + \int_0^t\!ds\,\Big[ }
  - \frac{\eps^{2\gamma}}{2}  \|\jmath^\eps \|_{L_2(\bb T)}^2
       \langle \vartheta''(u ) \nabla u , 
              [a'(u )]^2 \nabla u \rangle\Big]
  \\ & & \qquad
- \frac{ \eps^{2\gamma}}{2} \int_0^t\!ds\,
   \big\langle a(u )^2 \big[\vartheta''(u ) \nabla u  
   +\partial_x \vartheta'(u ) \big], \vartheta''(u ) \nabla u  
   +\partial_x \vartheta'(u ) \big\rangle
\Big\}
\end{eqnarray*}
$E^{\eps;\vartheta}$ is a continuous strictly positive $\bb
P^\eps$-supermartingale starting at $1$. For $\ell>0$ let
\begin{equation*}
B^{\ell}:=\{u \in \mc X \cap L_2\big([0,T];H^1(\bb T)\big)\,:\:
\| \nabla u \|_{L_2([0,T] \times \bb T)}^2 \le \ell
\}
\end{equation*}
Recall that $\mc W$ is the set of weak solutions to
\eqref{e:1.3}. Given a Borel subset $A \subset \mc W$ we have, for
$C$, $\eps_0$ as in Lemma~\ref{l:supexpnabla} (applied with $E^\eps
\equiv 0$) and $\ell>C$, $\eps \le \eps_0$
\begin{eqnarray}
\label{e:varanonlin}
\nonumber
\bb P^\eps(A)& \le &
\bb E^{P^\eps}\big(E^{\eps;\frac{\vartheta}{\eps^{2\gamma-1}}}(T,u) 
   [E^{\eps;\frac{\vartheta}{\eps^{2\gamma-1}}}(T,u)]^{-1} 
  \id_{A \cap B^{\ell/\eps}}(u)\big) +
 \bb P^\eps(B^{\ell/\eps})
\\
&\le & \sup_{u \in A \cap B^{\ell/\eps}} 
   [E^{\eps;\frac{\vartheta}{\eps^{2\gamma-1}}}(T,v)]^{-1}
  + \exp\big(-\frac{(\ell-C)^2}{C \eps^{2\gamma-1}(\ell+1)} \big)
\end{eqnarray}
where in the last line we used the supermartingale property of
$E^{\eps;\vartheta}$ and Lemma~\ref{l:supexpnabla}. Since
\begin{eqnarray*}
& & \eps^{2\gamma-1}\log E^{\eps;\frac{\vartheta}{\eps^{2\gamma-1}}}(T,u) = 
  -\int\!dx\, \vartheta(u_0(x),0,x)
\\
& &\qquad \qquad
  -\int \! ds\,dx\,
  \big[\big(\partial_s \vartheta)\big(u (s,x),s,x \big) 
  + \big(\partial_x Q\big)\big(u (s,x),s,x \big)\big]
\\
  & &\qquad \qquad 
  +
 \frac{\eps}2 \langle \langle  \vartheta''(u ) \nabla u ,
  D(u ) \nabla u  \rangle \rangle
  +\frac{\eps}2 \langle \langle\partial_x \vartheta'(u ),
  D(u ) \nabla u \rangle \rangle
  \\ & &\qquad \qquad
  - \frac{\eps^{2\gamma}}{2} \|\nabla \jmath^\eps \|_{L_2(\bb T)}^2
    \langle \langle \vartheta''(u ) a(u ), a(u )\rangle \rangle
\\
& &\qquad \qquad 
  - \frac{\eps^{2\gamma}}{2}  \|\jmath^\eps \|_{L_2(\bb T)}^2
       \langle \langle \vartheta''(u ) \nabla u , 
              [a'(u )]^2 \nabla u \rangle \rangle
\\
& &\qquad \qquad 
  - \frac{ \eps}{2} \langle \langle
    a(u)^2 \vartheta''(u ) \nabla u, \vartheta''(u ) \nabla u
                   \rangle \rangle
 -  \frac{ \eps}{2} \langle \langle a(u)^2 \partial_x \vartheta'(u ), 
                              \partial_x \vartheta'(u) \rangle \rangle
\\
& &\qquad \qquad
 -\eps  \langle \langle a(u)^2\vartheta''(u ) \nabla u,
               \partial_x \vartheta'(u ) \rangle \rangle
\end{eqnarray*}
by Cauchy-Schwartz inequality, for each $u \in B^{\ell/\eps}$
\begin{eqnarray}
 \label{e:uppboundineq1}
\nonumber
& &
 \eps^{2\gamma-1}\log E^{\eps;\frac{\vartheta}{\eps^{2\gamma-1}}}(T,u)
\ge 
  -\int\!dx\, \vartheta(u_0(x),0,x)
\\ \nonumber
& & \qquad
  -\int \! ds\,dx\,
  \big[\big(\partial_s \vartheta)\big(u (s,x),s,x \big) 
  + \big(\partial_x Q\big)\big(u (s,x),s,x \big)\big]
\\ \nonumber
& & \qquad
 +
 \frac{\eps}2 \langle \langle  \vartheta''(u ) \nabla u ,
  \big(D(u ) -a(u)^2\vartheta''(u)\big)\nabla u  \rangle \rangle
-C_\vartheta \sqrt{\eps \ell}
\\
& & \qquad
- C_{\vartheta} \eps^{2\gamma} \|\nabla \jmath^\eps \|_{L_2(\bb T)}^2
-
 C_\vartheta \eps^{2\gamma-1}\ell \|\jmath^\eps \|_{L_2(\bb T)}^2
   -  C_\vartheta \eps
 -\sqrt{\eps \ell} C_\vartheta
\end{eqnarray}
for a suitable constant $C_\vartheta >0 $ depending only on
$\vartheta$, $D$ and $a$. The key point now is that, if the
entropy sampler $\vartheta$ satisfies
\begin{eqnarray}
\label{e:thetaineq}
a(u)^2 \vartheta''(u,t,x) \le D(u) \qquad 
 \forall\,u\in [0,1],\, t \in [0,T],\,x \in \bb T
\end{eqnarray}
then the term $\langle \langle \vartheta''(u ) \nabla u , \big(D(u )
-a(u)^2\vartheta''(u)\big)\nabla u \rangle \rangle$ in
\eqref{e:uppboundineq1} is positive.
Namely, the largest term in the quadratic variation of
$N^{\eps;\vartheta}$ is controlled by the positive parabolic term
associated with the deterministic diffusion. Therefore taking the limit $\eps
\to 0$ in \eqref{e:uppboundineq1}, by the hypotheses assumed on $\jmath^\eps$,
for each entropy sampler $\vartheta$ satisfying \eqref{e:thetaineq} and each $u
\in B^{\ell/\eps}$
\begin{eqnarray}
 \label{e:uppboundineq2}
\nonumber
& & \varlimsup_\eps \eps^{2\gamma-1}\log
E^{\eps;\frac{\vartheta}{\eps^{2\gamma-1}}}(T,u)
\ge  
  -\int\!dx\, \vartheta(u_0(x),0,x) 
\\
& &
\qquad \qquad
  -\int \! ds\,dx\,
  \big[\big(\partial_s \vartheta)\big(u (s,x),s,x \big) 
  + \big(\partial_x Q\big)\big(u (s,x),s,x \big)\big]
\end{eqnarray}
We now take the logarithm of \eqref{e:varanonlin} and multiply it by
$\eps^{2\gamma-1}$. Taking the limits $\eps \to 0$, then $\ell \to
+\infty$, and using \eqref{e:uppboundineq2}, we have for each $\vartheta$
satisfying \eqref{e:thetaineq}
\begin{eqnarray*}
& &
\varlimsup_{\eps} \eps^{2\gamma-1} \log \bb P^\eps(A)
\le  - \inf_{u \in A} \big\{
  -\int\!dx\, \vartheta(u_0(x),0,x)
\\ & & \quad
  -\int \! ds\,dx\,
  \big[\big(\partial_s \vartheta)\big(u (s,x),s,x \big) 
  + \big(\partial_x Q\big)\big(u (s,x),s,x \big)\big]
\big\} \le - \inf_{u \in A} \sup_{\vartheta} P_{\vartheta,u}
\end{eqnarray*}
where we have applied the definition \eqref{e:2.6b} of
$P_{\vartheta,u}$.  Note that the map $\mc X \ni u \mapsto
P_{\vartheta,u} \in \bb R$ is lower semicontinuous.  Applying the
minimax lemma, we gather for a compact set $K \subset \mc W$
\begin{eqnarray*}
\varlimsup_{\eps} \eps^{2\gamma-1} \bb P^\eps(K)
 \le - \inf_{u \in K} \sup_{\vartheta} P_{\vartheta,u}
\end{eqnarray*}
where the supremum is taken over the entropy samplers $\vartheta$
satisfying \eqref{e:thetaineq}. It is easy to see that a weak solution
$u$ to \eqref{e:2.4} such that $\sup_{\vartheta}
P_{\vartheta,u}<+\infty$ is indeed an entropy-measure solution $u \in
\mc E$, and $\sup_{\vartheta} P_{\vartheta,u}=H(u)$.  
\end{proof}

\begin{proof}[Proof of Theorem~\ref{t:ld2}: lower bound]
  We will use the entropy method suggested by Lemma~\ref{l:ldglim}, as
  we did in the proof of Theorem~\ref{t:ld1} item (ii). Recall the
  Definition~\ref{d:splittable} of $\mc S$. Given $v \in \mc S$, we
  need to show that there exists a sequence $\{\bb Q^{\eps;v}\}
  \subset \mc P(\mc X)$ such that $\varlimsup
  \eps^{2\gamma-1}\mathrm{Ent}(\bb Q^{\eps;v}|\bb P^\eps) \le H(u)$
  and $\bb Q^\eps \to \delta_v$ in $\mc P(\mc X)$. The lower bound
  with rate $\upbar{H}$ then follows by a standard diagonal argument.

With minor adaptations from Theorem~2.5 in \cite{BBMN}, we have
that the following statement holds.
\begin{lemma}
\label{l:bbmn}
For each sequence $\beta_\eps \to
0$ and each $v \in \mc S$, there exist a sequence $\{w^\eps\}
\subset \mc X \cap L_2\big([0,T];H^1(\bb T)\big)$ and a sequence $\{\Psi^\eps\}
\subset L_2\big([0,T];H^2(\bb T)\big)$ such that:
\begin{itemize}
\item[\rm{(a)}]{ $w^\eps \to v$ in $\mc X$, and $w^\eps(0,x)=u_0(x)$.}
\item[\rm{(b)}]{ $\eps \|\nabla w^\eps \|_{L_2([0,T] \times \bb T)}^2
  \le C$ for some $C>0$ independent of
    $\eps$.}
\item[\rm{(c)}]{ $\varlimsup_\eps \frac{\eps^{-1}}{2} \langle \langle
    a(w^\eps)^2 \nabla \Psi^\eps, \nabla \Psi^\eps \rangle \rangle =
    H(v)$.}
\item[\rm{(d)}]{$\beta_\eps \big\| \nabla [a(w^\eps)\nabla
    \Psi^\eps]\big\|_{L_2([0,T] \times \bb T)}^2 \le
    C\,\eps^{-1}$, for some $C>0$ independent of $\eps$.}
\item[\rm{(e)}]{The equation
\begin{eqnarray*}
  \partial_t w^\eps +\nabla \cdot f(w^\eps)
  -\frac{\eps}2\nabla \cdot\big(D(w^\eps)\nabla w^\eps\big) =- \nabla \cdot 
  \big(a(w^\eps)^2\,\nabla \Psi^\eps\big)
\end{eqnarray*}
holds weakly.}
\end{itemize}
\end{lemma}

We let $\beta_\eps:= \eps^{-3/2} \|\jmath^\eps -\id \|_{W^{-1,1}(\bb
  T)}$, and let $\{w^\eps\}$, $\{\Psi^{\eps}\}$ be chosen
correspondingly. Note that with this choice of $\beta_\eps$ and by the
assumption on $\|\jmath^\eps -\id\|_{W^{-1,1}(\bb T)}$
\begin{eqnarray}
\label{e:beta}
\lim_\eps \eps^{-2}\int_0^t\!ds\,\|\jmath^\eps \ast \jmath^\eps
\ast [a(w^\eps)\nabla \Psi^{\eps}] -  a(w^\eps)\nabla \Psi^\eps 
                               \|_{L_2(\bb T)}^2=0
\end{eqnarray}
We define the martingale $M^{\eps;v}$ on $\Omega$ as
\begin{eqnarray*}
  M^{\eps;v}(t):=\eps^{-\gamma}\int_0^t \langle \jmath^\eps \ast 
       [ a(w^\eps) \nabla \Psi^{\eps}]  ,dW \rangle
\end{eqnarray*}
Then by Young inequality for convolutions:
\begin{eqnarray}
\label{e:quadminH}
\frac 12 \big[M^{\eps;v},M^{\eps;v} \big](T) \le
\frac{\eps^{-2\gamma}}2 \langle \langle   a(w^\eps)^2 \nabla \Psi^\eps,
\nabla \Psi^\eps
\rangle \rangle
\end{eqnarray}
In particular the stochastic exponential of $N^{\eps;v}$ is a
martingale on $\Omega$, and we can define the probability measure
$Q^{\eps;v} \in \mc P(\Omega)$ as
\begin{eqnarray*}
 Q^{\eps;v}(d\omega):= \exp
\big(N^{\eps;v}(T,\omega)-\frac 12 \big[N^{\eps;v},N^{\eps;v} \big](T,\omega)
\big) P(d\omega)
\end{eqnarray*}
and $\bb Q^{\eps;v}:=Q^{\eps;v} \circ (u^\eps)^{-1} \in \mc P(\mc X)$,
where $u^\eps:\Omega \to \mc X$ is the solution to
\eqref{e:2.2}. Reasoning as in \eqref{e:entropy1}, and using
\eqref{e:quadminH} and property (c) in Lemma~\ref{l:bbmn}
\begin{eqnarray}
\label{e:entropy2}
\nonumber
\varlimsup_\eps \eps^{2\gamma-1}\mathrm{Ent}(\bb Q^{\eps;v}|\bb P^{\eps;v})
& \le & \varlimsup_\eps \eps^{2\gamma-1}\mathrm{Ent}(Q^{\eps;v}| P)
\\ \nonumber
& = &
 \varlimsup_\eps \frac{\eps^{2 \gamma-1}}2
\int Q^{\eps;v}(d\omega) [M^{\eps;v},M^{\eps;v}](T,\omega)
\\ & \le & 
\varlimsup_\eps  \frac{\eps^{-1}}{2} \langle \langle
    a(w^\eps)^2 \nabla \Psi^\eps, \nabla \Psi^\eps \rangle \rangle = H(v)
\end{eqnarray}

We next need to prove that $\bb Q^{\eps;v}$ converges to $\delta_v$ in
$\mc P( \mc X)$ as $\eps \to 0$. By Girsanov theorem $\bb Q^{\eps;v}$
is a martingale solution to the stochastic Cauchy problem in the
unknown $u$
\begin{eqnarray}
\label{e:stoc4}
\nonumber
& & d u  =  \big[- \nabla \cdot f(u)
+\frac{\eps}{2} \nabla \cdot \big(D(u) \nabla
u\big)
-\nabla \cdot a(u) (\jmath \ast \jmath \ast(a(w^\eps) \nabla \Psi^\eps)
\big]\,dt 
\\ \nonumber 
& & \phantom{d u  =  }
+\eps^{\gamma}\, \nabla \cdot
\big[a(u) (\jmath^\eps \ast d W)\big]
\\ 
& & u(0,x) =  u_0^\eps(x)
\end{eqnarray}
In view of property (a) in Lemma~\ref{l:bbmn}, it is enough to check
that Lemma~\ref{l:stab} holds with
$v^\eps(u)(t,x)=w^\eps(t,x)$. Indeed, still by property (a) in
Lemma~\ref{l:bbmn} and the assumptions of this theorem, conditions (i)
and (ii) in Lemma~\ref{l:stab} are immediate. By property (e) in
Lemma~\ref{l:bbmn} and \eqref{e:stoc4}, $\bb Q^{\eps;v}$ is a
martingale solution to \eqref{e:stocgir} with $G_1^\eps \equiv 0$,
\begin{equation*}
G_2^\eps(u,w)=a(w) \big[\jmath^\eps \ast \jmath^\eps \ast
[a(w^\eps)\nabla \Psi^{\eps}] - a(w^\eps)\nabla \Psi^\eps\big]
\end{equation*}
\begin{equation*}
G_3^\eps(u,w)=[a(w)-a(u)] [\jmath \ast \jmath \ast(a(w^\eps) \nabla
\Psi^\eps]
\end{equation*}
Therefore, in view of \eqref{e:beta}, condition (iii) in
Lemma~\ref{l:stab} is easily seen to hold. Condition (iv) is also
immediate from the definition of $G_3$ and the bound on $\bb
Q^{\eps;v}\big(\eps \|\nabla u\|_{L_2([0,T]\times \bb T)} >\ell \big)$
provided by the application of Lemma~\ref{l:supexpnabla} for $\bb
P^\eps$ (thus with $E^\eps \equiv 0$), the entropy bound
\eqref{e:entropy2}, and the usual entropy inequality.  
\end{proof}

%\renewcommand{\theequation}{A-\arabic{equation}}
% redefine the command that creates the equation no.
%\setcounter{equation}{0}  % reset counter 
%\section*{APPENDIX: Existence and uniqueness results for 
%fully nonlinear parabolic SPDEs with conservative noise}
\appendix
\section{Existence and uniqueness results for 
fully nonlinear parabolic SPDEs with conservative noise}
\label{s:A}

In this appendix, we are concerned with existence and uniqueness
results for the Cauchy problem in the unknown $u\equiv u(t,x)$, $t \in
[0,T]$, $x \in \bb T$
\begin{eqnarray}
\label{e:A1}
\nonumber
& & du =  \big[- \nabla \cdot f(u) 
     +\frac 12 \nabla \cdot \big( D(u) \nabla u \big) \big]\, dt 
      + \nabla \cdot \big[a(u) (\jmath \ast dW) \big]
\\
& & u(0,x) =  u_0(x)
\end{eqnarray}
Although we assume the space-variable $x$ to run on a one-dimensional
torus $\bb T$, it is not difficult to extend the results given below
to the case $x \in \bb T^d$ or $x \in \bb R^d$ for $d \ge 1$.

Let $W$ be an $L_2(\bb T)$--valued cylindrical Brownian motion on a
given standard filtered probability space $\big(\Omega, \mf F, \{\mf
F_t\}_{0\le t \le T}, P\big)$. Hereafter we set
\begin{eqnarray*}
Q(v):= a'(v)^2 \| \jmath\|_{L_2(\bb T)}^2
\end{eqnarray*}
We will assume the following hypotheses:
\begin{itemize}
\item[\textbf{A1)}] {$f$ and $D$ are uniformly Lipschitz on $\bb R$.}
\item[\textbf{A2)}] {$a \in C^2(\bb R)$ is uniformly bounded.}
\item[\textbf{A3)}] {$\jmath \in H^1(\bb T)$ and, with no loss of
    generality, $\int\!dx\,|\jmath(x)|=1 $. }
\item[\textbf{A4)}] {There exists $c>0$
    such that $D \ge Q+c$.}
\item[\textbf{A5)}]{$u_0:\Omega \to L_2(\bb T)$ is $\mf F_0$-Borel measurable
and satisfies $\bb E^{\bb P} \big( \|u_0\|_{L_2(\bb T)}^2 \big)<+\infty$.}
\end{itemize}

We introduce the Polish space $Y:= C\big([0,T];H^{-1}(\bb T)\big) \cap
L_2\big([0,T];H^1(\bb T) \big) \cap L_{\infty}\big([0,T]; L_2(\bb
T)\big)$. A probability measure $\upbar{\bb P}$ on $Y$ is a
\emph{martingale solution} to \eqref{e:A1} iff the law of $u(0)$ under
$\upbar{\bb P}$ is the same of the law of $u_0$, and for
each $\varphi \in C^\infty\big([0,T] \times \bb T \big)$
\begin{eqnarray}
\label{e:A2}
\nonumber
& & \langle M(t,u), \varphi \rangle :=      \langle u(t),\varphi(t)\rangle -
\langle u(0), \varphi(0) \rangle
\\ & & \phantom{\langle M(t,u), \varphi \rangle :=}
     - \int_0^t\!ds\, \langle u,\partial_s \varphi \rangle 
        + \langle f(u) -\frac 12 D(u) \nabla u,\nabla \varphi \rangle
\end{eqnarray}
is a continuous square-integrable martingale with respect to
$\upbar{\bb P}(du)$ with quadratic variation
\begin{eqnarray}
\label{e:A3}
\big[\langle M, \varphi \rangle ,\langle M, \psi \rangle\big](t,u)
= \int_0^t\!ds\, \langle \jmath \ast (a(u) \nabla \varphi),
               \jmath \ast (a(u) \nabla \psi) \rangle
\end{eqnarray}
We say that a progressively measurable process $u:\Omega \to Y$ is a
\emph{strong solution} to \eqref{e:A1} iff $u(0)=u_0$ $P$-a.s.\ and
for each $\varphi \in C^\infty \big([0,T]\times \bb T\big)$
\begin{eqnarray}
\langle M, \varphi \rangle = -\int_0^t
      \langle \jmath \ast \big(a(u) \nabla \varphi\big) , dW \rangle
\end{eqnarray}

In this appendix we prove
\begin{theorem}
\label{t:exun}
Assume \textbf{A1)}--\textbf{A5)}. Then there exists a unique strong
solution $u$ to \eqref{e:A1} in $Y$. Such a solution $u$ admits a
version in $C\big([0,T];L_2(\bb T)\big)$. Furthermore, if $u_0$ takes
values in $[0,1]$ and $a$ is supported by $[0,1]$, then $u$ takes
values in $[0,1]$ a.s..
\end{theorem}
By compactness estimates we will prove that there exists a solution to
the martingale problem related to \eqref{e:A1}. Then we will provide
pointwise uniqueness for \eqref{e:A1} using a stability result similar
to the one used in the proof of Lemma~\ref{l:stab}. By Yamada-Watanabe
theorem we get the existence and uniqueness stated in
Theorem~\ref{t:exun}. We remark that assumption \textbf{A4)} is a key
hypotheses in the proof of Theorem~\ref{t:exun}, as it implies that
the noise term is smaller than the second order parabolic term, thus
allowing some a priori bounds. In general, one may expect nonexistence
of the solution to \eqref{e:A1} if such a condition fails, see
\cite[Example~7.21]{DZ}.

\begin{lemma}
\label{l:piecestoc}
Let $0 \le t^\prime <t^{\prime \prime} \le T$, let $u^\prime, v:\Omega
\to L_2(\bb T)$ be $\mf F_{t^\prime}$-measurable maps such that $\bb
E^{P}\big( \| |u^\prime|+ |v|+|\nabla v| \|_{L_2(\bb
  T)}^2\big)<+\infty $. Then the stochastic Cauchy problem in the
unknown $w$
\begin{eqnarray}
\label{e:piecestoc}
\nonumber
& & dw  =  \big[- \nabla \cdot f (w)
+\frac{1}{2} \nabla \cdot \big(D(v) \nabla w\big)\big]\,dt
+\nabla \cdot \big[a(v) (\jmath \ast d W)\big]
\\
&  &  w(t^\prime,x) = u^\prime(x)
\end{eqnarray}
admits a unique strong solution $u$ in $L_2\big([t^\prime,t^{\prime
  \prime}];H^1(\bb T) \big) \cap C\big([t^\prime,t^{\prime
  \prime}],H^{-1}(\bb T)\big)$ with probability $1$. For each $t\in
[t^\prime,t^{\prime \prime}]$, such a solution $u$ satisfies
\begin{eqnarray}
\label{e:pieceito}
\nonumber
& &
\langle u(t),u(t)\rangle +\int_{t^\prime}^t \!ds
 \langle D(v) \nabla u, \nabla u \rangle =  N(t,t^\prime)
+\langle u^\prime,u^\prime \rangle
\\
& & \qquad \qquad \qquad \quad
+\int_{t^\prime}^t \!ds\, \big[ \langle Q(v)
\nabla v, \nabla v \rangle+\|\nabla \jmath\|_{L_2(\bb T)}^2 \int \!dx\,a(v)^2
\big]
\end{eqnarray}
where $N(t,t^\prime):=-2 \int_{t^\prime}^t \langle \jmath \ast \big(
a(v) \nabla u \big), dW\rangle$. Furthermore 
\begin{equation*}
\bb E^{P} \big( \sup_{t \in
  [t^\prime,t^{\prime \prime}]} \|u(t)\|_{L_2(\bb T)}^2 \big) <+\infty
\end{equation*}
\end{lemma}
\begin{proof}
  Existence and uniqueness of the semilinear equation
  \eqref{e:piecestoc} are standard, see e.g.\
  \cite[Chap.~7.7.3]{DZ}. Applying It\^{o} formula to the function
  $L_2(\bb T) \ni w \mapsto \langle w, w \rangle \in \bb R$ we get
  \eqref{e:pieceito}. Note that by Burkholder-Davis-Gundy inequality
  \cite[Theorem~4.4.1]{RY}, Young and Cauchy-Schwarz inequalities, for
  suitable constants $C,C^{\prime}>0$
\begin{eqnarray*}
  \bb E^{P} \big(\sup_{t\in [t^\prime,t^{\prime \prime}]} |N(t,t^\prime)|\big) 
        & \le &
  C\,  \bb E^{P} \big( \big[N(\cdot,t^\prime),
  N(\cdot,t^\prime)\big](t^{\prime \prime})^{1/2} \big)
\\
 & = &
  2\,C\,  \bb E^{P}\Big( 
  \| \jmath \ast \big(a(v) \nabla u \big) 
    \|_{L_2([t^{\prime},t^{\prime \prime}]\times \bb T)}\Big)
\\
& \le &
   2\,C\,  \bb E^{P}\Big( 
  \| \big(a(v) \nabla u \big) 
    \|_{L_2([t^{\prime},t^{\prime \prime}]\times \bb T)}\Big)
\\
  &\le & 
C^{\prime}\,  \Big[ \bb E^{P} \Big(\int_{t^\prime}^{t^{\prime \prime}}\!ds\, 
  \langle D(v) \nabla u, \nabla u \rangle
\Big) \Big]^{1/2} 
\end{eqnarray*}
so that the bound on $\bb E^{P} \big(\sup_{t \in [t^\prime,t^{\prime
    \prime}]} \|u(t)\|_{L_2(\bb T)}^2\big)$ is easily obtained by
taking the supremum over $t$ and the $\bb E^{P}$ expected values in
\eqref{e:pieceito}.  
\end{proof}

We next introduce a sequence $\{u^n\}$ of adapted processes in $Y$. We
will gather existence of a weak solution to \eqref{e:A1} by tightness
of the laws $\{\bb P^n\}$ of such a sequence.

For $n \in \bb N$ and $i=0,\ldots,\,2^n$ let $t_i^n:=i2^{-n} T$, and
let $\{\imath^n\}$ be a sequence of smooth mollifiers on $\bb T$ such
that $\lim_{n} 2^{-n} \|\imath^n\|_{L_1(\bb T)}^2 =0$. We define a
process $u^n$ on $Y$ and the auxiliary random functions
$\{v_i^n\}_{i=0}^{2^n}$ on $\bb T$ as follows. For $i=0$ we set
\begin{eqnarray*}
& & u^n(0) :=  u_0
\\
& & v_0^n := \imath^n \ast u_0
\end{eqnarray*}
and for $i=1,\ldots,2^n-1$ and $t \in [t_i^n,t_{i+1}^n]$, we let
$u^n(t)$ be the solution to the problem \eqref{e:piecestoc} with
$u^\prime=u(t_i^n)$ and $v=v_i^n$, where for $i\ge 1$ we set
\begin{equation}
\label{e:vin}
v_i^n:=\frac{2^n}{T} \int_{t_{i-1}^n}^{t_i^n}\!ds\, u^n(s)
\end{equation}
By Lemma~\ref{l:piecestoc}, these definitions are well-posed, and
$u^n$ is in $Y$ with probability $1$. We also define a sequence
$\{v^n\}$ of cadlag processes in the Skorohod space
$D\big([0,T);L_2(\bb T)\big)$, by requiring 
\begin{equation}
\label{e:vvern}
  v^n(t)=v_i^n \text{for $t \in
    [t_i^n,t_{i+1}^n)$}
\end{equation}

\begin{lemma}
\label{l:tight}
There exists a constant $C>0$ independent of $n$ such that
\begin{eqnarray}
\label{e:tight1} 
\bb E^{P} \Big( \sup_{t\in[0,T]} \| u^n(t)\|_{L_2(\bb T)}^2
+\| \nabla u^n\|_{L_2([0,T] \times \bb T)}^2 \Big) \le C
\end{eqnarray}
and for each $\varphi \in H^1(\bb T)$ such that $\| \nabla
\varphi \|_{L_2(\bb T)}^2 \le 1$, for each $\delta>0$ and $r\in(0,1)$
\begin{eqnarray}
\label{e:tight2} 
P \big( \sup_{s,t \in [0,T]\, : |s-t|\le \delta} 
           \big|\langle u^n(t)-u^n(s),\varphi\rangle \big|>r \big) 
\le C\,\delta\,r^{-2}
\end{eqnarray}

Furthermore for each $r>0$
\begin{eqnarray}
\label{e:uclosev}
  \lim_{n \to \infty} P\big(
   \| u^n-v^n \|_{L_2([0,T] \times \bb T)}  > r
 \big) =0
\end{eqnarray}

\end{lemma}
\begin{proof}
  Writing It\^{o} formula \eqref{e:pieceito} for $u^n$ in the
  intervals $[t_i^n,t_{i+1}^n]$ and summing over $i$, we get for each
  $t \in [0,T]$
\begin{eqnarray*}
& &
\langle u^n(t),u^n(t)\rangle +\int_0^t\!ds\,
 \langle D(v^n) \nabla u^n, \nabla u^n \rangle = \langle u_0,u_0 \rangle
\\ & &
\qquad \qquad \qquad 
+\int_0^t\! ds\,\big[ \langle Q(v^n)
\nabla v^n, \nabla v^n \rangle+\|\nabla \jmath\|_{L_2(\bb T)}^2 
\int \!dx\, a(v^n)^2 \big]+ N^n(t)
\end{eqnarray*}
where, by the same means of Lemma~\ref{l:piecestoc} and Doob's
inequality, the martingale
\begin{eqnarray*}
N^n(t):=2 \int_0^t \langle \jmath \ast \big( a(v^n) \nabla u^n
\big), dW\rangle
\end{eqnarray*}
enjoys the bound
\begin{eqnarray*}
  \bb E^{P} \big( \sup_{s \in [0,T]} |N^n(t)|^2 \big)
  \le C_1 \,  \bb E^{P}\big( \| \nabla u^n\|_{L_2([0,T] \times \bb T)}^2 \big)
\end{eqnarray*}
for some $C_1>0$ depending only on $D$ and $a$. Note that, by the
definition of $v_i^n$ \eqref{e:vin}, hypotheses
\textbf{A4)}-\textbf{A5)} and Young inequality for convolutions
\begin{eqnarray*}
& &
\int_0^t \!ds\,\langle Q(v^n) \nabla v^n, \nabla v^n\rangle 
\\  & & \qquad \qquad 
 \le 
C_2 \int_0^{t_1^n}\!ds\, \| \imath^n \ast u_0\|_{L_2(\bb T)}^2
+ \int_0^t \!ds\,\langle Q(v^n) \nabla u^n,\nabla u^n \rangle 
\\ & & \qquad \qquad 
 \le 2^{-n}T\,C_2\,\|\imath^n\|_{L_1(\bb T)}^2  \| u_0\|_{L_2(\bb T)}^2 +
\int_0^t \!ds\,\langle (D(v^n)-c) \nabla u^n,\nabla u^n \rangle
\end{eqnarray*}
for some constant $C_2$ depending only on $a$. Patching all together
\begin{eqnarray*} 
& &   \bb E^{P} \big( \sup_{t \in [0,T]} \| u^n(t)\|_{L_2([0,T] \times \bb T)}^2
 +c  \langle \langle D(v^n) \nabla u^n, \nabla u^n \rangle\rangle  \big)
\\
& & \qquad
\le
    \big(1+2^{-n}T\,C_2 \|\imath^n\|_{L_1(\bb T)}^2 \big) \,
                \bb E^{P}\big( \| u_0\|_{L_2(\bb T)}^2 \big)
\\ & &  \qquad \phantom{\le}
+  C_1\,  \bb E^{P}\big( \langle \langle  D(v^n) \nabla u^n, 
         \nabla u^n \rangle\rangle^{1/2} \big)
+\|\nabla \jmath\|_{L_2(\bb T)}^2\,  \bb E^{P}
\big(\| a(v^n)\|_{L_2([0,t] \times \bb T)}^2 \big)
 \end{eqnarray*}
 Since $2^{-n}\|\imath^n\|_{L_1(\bb T)}$ was assumed bounded,
 and since the last term in the right hand side is bounded uniformly
 in $n$, it is not difficult to gather \eqref{e:tight1}.

Since $u$ satisfies \eqref{e:piecestoc} in each interval
$[t_i^n,t_{i+1}^n]$
\begin{eqnarray*}
  \big| \langle u^n(t)-u^n(s),\varphi \rangle \big|
  & \le & C_3 \big( 1+ \|\nabla u^n\|_{L_2([0,T] \times \bb T)} \big) 
 \| \nabla \varphi \|_{L_2(\bb T)}|t-s|^{1/2}
\\ & &
+ \big|\int_s^t \langle \jmath \ast (a(v) \nabla \varphi), 
                                               dW \rangle \big|
\end{eqnarray*}
for a suitable constant $C_3$ depending only on $f$ and
$D$. \eqref{e:tight2} then follows from the first part of the lemma.

Since $v^n(t)= \imath^n \ast u_0$ for $t\in [0,t_1^n)$, the bound
\eqref{e:tight1} implies
\begin{equation*}
  \lim_{n\to \infty }
 P\big(\|u^n-v^n\|_{L_2([0,t_1^n]\times \bb T)}> r \big)=0
 \end{equation*}
 for each $r>0$. Therefore, still by \eqref{e:tight1}, in order to
 prove \eqref{e:uclosev}, it is enough to show that for each
 $r,\,\ell>0$
\begin{eqnarray*}
  \lim_{n \to \infty}  P\big(
   \| u^n-v^n\|_{L_2([t_1^n,T] \times \bb T)} > r,
 \|\nabla u^n\|_{L_2([0,T] \times \bb T)}^2 \le \ell
 \big) =0
\end{eqnarray*}
Let $\kappa \in C^\infty(\bb T)$ be such that $\int\,dx\,
\kappa(x)=1$, and that
\begin{eqnarray}
\label{e:knorm}
\nonumber
& &
\|\kappa-\mathrm{id}\|_{-1,1}:= \sup \big\{
\int\!dx\, \big|\int\!dy\, \kappa(x-y)\varphi(y)-\varphi(x)
\big|,
\\ & &
\phantom{\|\kappa-\mathrm{id}\|_{-1,1}:= \sup \Big\{
}
\varphi \in C^\infty(\bb T)\,:\:\sup_x |\nabla \varphi(x)|\le 1
\Big\} \le \frac{r}{2\ell}
\end{eqnarray}
It is immediate to see
that such a $\kappa$ exists. Then
\begin{eqnarray*}
& & \|u^n-v^n\|_{L_2([t_1^n,T] \times \bb T)}
 \le
 \|u^n-\kappa \ast u^n\|_{L_2([t_1^n,T] \times \bb T)}
\\ & & \qquad \phantom{\le}
  +  \|v^n-\kappa \ast v^n\|_{L_2([t_1^n,T] \times\bb T)}
  +  \|\kappa \ast u^n-\kappa \ast v^n\|_{L_2([t_1^n,T] \times\bb T)}
\\ &  &  \qquad \le 
  \|\kappa-\mathrm{id}\|_{-1,1}  \big[
      \| \nabla u^n \|_{L_2([t_1^n,T] \times\bb T)} +
      \| \nabla v^n \|_{L_2([t_1^n,T] \times\bb T)} \big]
\\ & & \qquad \phantom{\le}
  +  \| \kappa \ast (u^n- v^n)\|_{L_2([t_1^n,T] \times \bb T)}
\end{eqnarray*}
where in the last inequality we used the Young inequality. 
By the definition \eqref{e:vin}-\eqref{e:vvern} of $v^n$, $\| \nabla
v^n \|_{L_2([t_1^n,T] \times \bb T)}^2 \le \| \nabla u^n \|_{L_2([0,T]
  \times\bb T)}^2$. Moreover
\begin{eqnarray*} 
& & 
\int_{t_1^n}^T\!dt\, \| \kappa \ast (u^n -v^n)\|_{L_2(\bb T)}^2
\\ & & \quad
 = \sum_{i=1}^{2^n-1} 
   \int_{t_i^n}^{t_{i+1}^n} \!dt\,
         \Big\| \kappa \ast u^n(t) -
\frac{2^{n}}{T}
          \int_{t_{i-1}^n}^{t_i^n}\!ds\,
\kappa \ast u^n(s) \Big\|_{L_2(\bb T)}^2
 \\ & & \quad
\le T\,\sup_{|t-s|\le 2^{-n+1}T} 
 \| \kappa \ast (u^n(t) -u^n(s))\|_{L_2(\bb T)}^2
\end{eqnarray*}
Therefore by \eqref{e:knorm}
\begin{eqnarray}
\label{e:filext}
\nonumber
& & \|u^n-v^n\|_{L_2([t_1^n,T] \times \bb T)}^2
\le 
\frac{r}{2\ell} %\int_{t_1^n}^T\!dt\,
      \| \nabla u^n \|_{L_2([t_1^n,T] \times \bb T)}^2
\\ & &  \qquad \qquad \qquad
+T\,\sup_{|t-s|\le 2^{-n+1}T} 
  \| \kappa \ast (u^n(t) -u^n(s))\|_{L_2(\bb T)}^2
\end{eqnarray}
so that
\begin{eqnarray*}
& &   \lim_{n \to \infty}  P\big(
   \| u^n-v^n\|_{L_2([t_1^n,T] \times \bb T)} > r,
 \|\nabla u^n\|_{L_2([0,T] \times \bb T)}^2 \le \ell
 \big)
\\ & &
\qquad  \le 
\varlimsup_{n \to \infty} P \big( \sqrt{T}\,\sup_{|t-s|\le 2^{-n+1}T} \|
\kappa \ast (u^n(t) -u^n(s))\|_{L_2(\bb T)} \ge r/2
\big)
\end{eqnarray*}
which vanishes in view of \eqref{e:tight2}.  
\end{proof}

We define $\bb P^n$ to be the law of $u^n$, namely $\bb P^n = P \circ
(u^n)^{-1}$. In order to establish tightness of the sequence $\{\bb
P^n\}$, the $\bb P^n$ will be regarded as probability measures on
$C\big([0,T],H^{-1}(\bb T)\big) \supset Y$, although they are
concentrated on $Y$.
\begin{corollary}
  \label{c:tight} 
  $\{\bb P^n\}$ is tight, and thus compact, on $C\big([0,T],H^{-1}(\bb
  T)\big)$ equipped with the uniform topology. Furthermore each limit
  point $\upbar{\bb P}$ of $\{\bb P^n\}$ is concentrated on $Y$ and
  satisfies
\begin{eqnarray}
\label{e:limbound}
   \bb E^{\upbar{\bb P}} \big(\sup_t \| u(t)\|_{L_2(\bb T)}^2
  + 
    \| \nabla u\|_{L_2([0,T] \times \bb T)}^2 \big)<+\infty
\end{eqnarray}
\end{corollary}
\begin{proof}
By the compact Sobolev embedding of $L_2(\bb T)$ in $H^{-1}(\bb T)$, the
estimate  \eqref{e:tight1} implies that \emph{compact containment condition} is
satisfied, namely there exists a sequence $\{K_\ell\}$ of compact subsets of
$H^{-1}(\bb T)$ such that
\begin{eqnarray*}
\lim_\ell \varlimsup_n \bb P\big(\exists t \in [0,T]\,:\:u^n(t) \not\in K_\ell
\big)=0
\end{eqnarray*}
Moreover the estimate \eqref{e:tight2} implies that for each $\varphi
\in H^1(\bb T)$ the laws of the processes $t \mapsto \langle
u^n(t),\varphi \rangle$ are tight in $C\big([0,T];\bb R\big)$ as $n$
runs on $\bb N$, see \cite[page~83]{Bi}. By \cite[Theorem~3.1]{Ja},
we get tightness of $\{\bb P^n\}$ on $C\big([0,T],H^{-1}(\bb T)\big)$.

\eqref{e:limbound} follows immediately by \eqref{e:tight1}.

\end{proof}

The following statement is derived following closely the proof of
Proposition~3.5 in \cite{BBMN}.
\begin{lemma}
\label{p:compactlow}
Let $K \subset C\big([0,T]; U\big)$. Suppose that
each $u \in K$ has a Schwartz distributional derivative in the
$x$-variable $\nabla u \in L_2([0,T]\times \bb T)$, and suppose that
exists $\zeta>0$ such that $\| \nabla u\|_{L_2([0,T]\times \bb T)} \le
\zeta$. Then $K$ is strongly compact in $\mc X$.
\end{lemma}

\begin{proposition}
\label{p:uexists}
Each limit point $\upbar{\bb P}$ of $\{\bb P^n\}$ is a weak solution to
\eqref{e:A1}.
\end{proposition}
\begin{proof}
  Let $\upbar{\bb P}$ be a limit point of $\{\bb P^n\}$ along a
  subsequence $n_k$. The law of $u(0)$ under $\upbar{\bb P}$ coincides
  with the law of $u_0$.  For $u \in Y$, $v \in D\big([0,T);L_2(\bb T)
  \big)$ and $\varphi \in C^\infty\big([0,T] \times \bb T\big)$ let
\begin{eqnarray*}
\langle M(t;u,v), \varphi \rangle & := &
     \langle u(t),\varphi(t)\rangle - \langle u(0), \varphi(0) \rangle
\\ & &
     - \int_0^t \!ds\, \big\langle u,\partial_t \varphi  \rangle
        - \langle f(v) -\frac 12 D(v) \nabla u,\nabla \varphi \big\rangle
\end{eqnarray*}
By \eqref{e:uclosev}, \eqref{e:tight1}, and Lemma~\ref{p:compactlow},
the law of $\langle M(\cdot;u^n,v^n), \varphi \rangle$ converges,
along the subsequence $n_k$, to the law of $\langle M(\cdot;u,u),
\varphi \rangle=\langle M(\cdot,u), \varphi \rangle$ under $\upbar{\bb
  P}$.

For each $n$ and $\varphi$, $\langle M(\cdot;u^n,u^n), \varphi
\rangle$ is a martingale with respect to $\bb P^n$, with quadratic
variation
\begin{equation*}
\big[\langle M(\cdot;u^n,u^n), \varphi \rangle,\langle
M(\cdot;u^n,u^n), \varphi \rangle\big](t) =\big\|
\jmath \ast (a(v^n) \nabla \varphi)\big\|_{L_2([0,t] \times \bb T)}^2
\end{equation*}
Still by \eqref{e:uclosev}, \eqref{e:tight1}, and
Lemma~\ref{p:compactlow}, we have that $\langle M(\cdot,u), \varphi
\rangle$ is a martingale under $\upbar{\bb P}$, with quadratic
variation given by \eqref{e:A3}.  
\end{proof}

\begin{proposition}
\label{p:uunique}
There exists at most one strong solution to \eqref{e:A1} in $Y$. Each
strong solution to \eqref{e:A1} admits a version in
$C\big([0,T];L_2(\bb T)\big)$.
\end{proposition}
\begin{proof}
  Let $u$, $v$ be to strong solutions to equation \eqref{e:A1}. By Ito
  formula, for $l \in C^2(\bb R)$ with bounded derivatives
\begin{eqnarray}
\label{e:Astima}
\nonumber
& & \int \!dx\, l(u-v)(t) - l(0)+ \frac{1}{2}
\int_0^t \!ds\,  \langle D(u) l^{\prime\prime}(u-v) \nabla (u-v), 
                              \nabla (u-v)\rangle 
\\ & & \nonumber
\qquad
= X(t) + \int_0^t \!ds \langle l^{\prime \prime}(u-v) \nabla (u-v),
                f(u) -f(v)\rangle
\\ & &\nonumber
 \qquad \phantom{=}
-\frac{1}{2} \int_0^t \!ds\, \langle l^{\prime \prime}(u-v) \nabla (u-v), 
                           [D(u)-D(v)] \nabla v \rangle
\\ & & \nonumber
\qquad \phantom{=}
+\frac{1}{2} \int_0^t \!ds\,
\langle  l^{\prime \prime}(u-v), \|\nabla \jmath\|_{L^2(\bb T)}^2 
         \big(a(u)-a(v) \big)^2  
\\ & &
\qquad \phantom{=+\frac{1}{2} \int_0^t \!ds\,
\langle  l^{\prime \prime}(u-v), }
  +\|\jmath \|_{L^2(\bb T)}^2 
  \big(a'(u) \nabla u-a'(v) \nabla v \big)^2 
 \rangle
\end{eqnarray}
and the quadratic variation of the martingale $X(t)$ enjoys the bound
\begin{equation*}
  \big[X,X \big](t) \le \int_0^t \!ds\,\|l^{\prime \prime}(u-v)
  \nabla(u-v)\big(a(u)-a(v)\big)\|_{L^2(\bb T)}^2
\end{equation*}
We next introduce the real number
\begin{eqnarray*}
R:=\Big[ \bb E^{P} \Big(\int_0^t \!ds\, \langle l^{\prime \prime}(u-v) \nabla
(u-v),\nabla(u-v) \rangle \Big) \Big]^{1/2}
\end{eqnarray*}
Taking the supremum over $t$ and the $\bb E^{P} $ expected value in
\eqref{e:Astima}, using repeteadly H\"{o}lder inequality and the
Burkholder-Davis-Gundy inequality \cite[Theorem~4.4.1]{RY},
assumptions \textbf{A2)} and \textbf{A5)} and the bound
\eqref{e:limbound}, we get for a suitable constant $C>0$
\begin{eqnarray*}
& &   \bb E^{P} \Big(  \sup_{t\le T} \int\!dx\,l(u-v)(t) \Big)+ c R^2 
\\ & &  \qquad \qquad
\le 2\,l(0)+ C \big[  \bb E^{P} \big( \|l^{\prime \prime}(u-v) |u-v|^2 
\|_{L^\infty([0,T]\times \bb T)}\big) \big]^{1/2} R 
\\  & & \qquad \qquad \phantom{\le}
+C  \bb E^{P} \Big( \int_0^t \!ds\,
\langle l^{\prime \prime}(u-v)|u-v|,|u-v|\rangle \Big)
\end{eqnarray*}
For any $\delta>0$, we can choose $l$ so that $|z|\le l(z) \le
|z|+\delta$, $l(z)=|z|$ for $|z|\ge \delta$, and $|l^{\prime
  \prime}(z)| \le 3 \delta^{-1}$. Therefore
\begin{eqnarray*}
  \bb E^{P} \big(  \sup_t \|u-v\|_{L^1(\bb T)} \big)
& \le &  \bb E^{P} \Big( \sup_{t} \int\!dx\, l(u-v)(t) \Big)
\\
& \le& 2 \delta-c R^2+ C \sqrt{\delta} R + C \delta 
%\\
%& \le & 
\le
    \Big(\frac{C^2}{4 c} + C+2\Big) \delta
\end{eqnarray*}
Since the last inequality holds for any $\delta>0$, we have $u=v$.

The $C\big([0,T];L_2(\bb T) \big)$ regularity for a version $u$ can be
easily derived from It\^{o} formula for the map $(t,u) \mapsto
\int\!dx\,u(t,x)^2$.  
\end{proof}

\begin{proof}[Proof of Theorem~\ref{t:exun}]
  Existence and uniqueness of a strong solution to \eqref{e:A1} is a
  consequence of Proposition~\ref{p:uexists},
  Proposition~\ref{p:uunique} and Yamada-Watanabe theorem
  \cite[Chap.~5, Corollary~3.23]{KS}. The fact that $u$ takes values
  in $[0,1]$ is provided in the same fashion of
  Lemma~\ref{l:supexpnabla}. Let $\{l^n\}$ be a sequence of infinitely
  differentiable convex functions on $\bb R$ with bounded
  derivatives. We can choose $\{l_n\}$ such that for $v \in [0,1]$
  $l_n^{\prime \prime}(v) \le D(v)\,a^{-2}(v)$ and $l_n(v) \le C_n
  (1+v^2)$ (for some $C_n>0$), while $l_n(v) \uparrow +\infty$ for
  $n\to +\infty$ pointwise for $v \not \in [0,1]$. By It\^{o} formula
\begin{eqnarray*}
& &\int\!dx\, \big[l_n(u(t))- l_n(u_0)\big] 
  + \frac{1}{2} \int_0^t \!ds\, 
      \big\langle l_n^{\prime \prime}(u) D(u)\,
             \nabla u,l_n^{\prime \prime}(u)\, \nabla u \big\rangle
\\
& & =\frac 12 \int_0^t \!ds\,\big\langle 
    l_n^{\prime \prime}(u) \nabla u, Q(u)\,\nabla u \big \rangle 
  +\|\nabla \jmath\|_{L_2(\bb T)}^2
  \int_0^t \!ds\int\!dx \, l_n^{\prime \prime}(u)\,
a(u)^2+N_n(t)
\end{eqnarray*}
where $N_n(t)$ is a martingale, and by Young inequality for
convolutions its quadratic variation is bounded by $\big[N_n , N_n
\big](t) \le \|a(u) l_n^{\prime \prime}(u) \, \nabla
u\|_{L_2([0,T]\times \bb T)}^2$. Following closely the proof of
Lemma~\ref{l:supexpnabla}, we gather for some constant $C$ independent
of $n$
\begin{eqnarray*}
   \bb E^{P} \big( \sup_{t \le T}  \int \!dx\, l_n(u(t)) \big) 
  \le  \bb E^{P} \big(  \int \!dx\, l_n(u_0)  \big)+ C
\end{eqnarray*}
As we let $n \to \infty$, the left hand side stays bounded, and since
$l_n \to +\infty$ pointwise off $[0,1]$, we have $dx\,d P$-a.s. that
$u(t,x) \in [0,1]$, for each $t \in [0,T]$.  
\end{proof}

\noindent\textit{Acknowledgements}
  I am grateful to Lorenzo Bertini for introducing me to the problem
  and providing invaluable help. I also thank S.R.S.\ Varadhan for
  enlightening discussions both on technical and general aspects of
  this work.  I acknowledge the hospitality and the support of
  Istituto Guido Castelnuovo (Sapienza Universit\`a di Roma), and
  Courant Institute of Mathematical Sciences (New York
  University). This work was partially supported by ANR LHMSHE.

\end{document}